\documentclass[10pt, reqno]{amsart}

\usepackage[utf8]{inputenc} 
\usepackage[T1]{fontenc}    
\usepackage{hyperref}       
\usepackage{url}            
\usepackage{booktabs}       
\usepackage{amsfonts}       
\usepackage{nicefrac}       
\usepackage{microtype}      
\usepackage{lipsum}
\usepackage{amsthm}
\usepackage{amssymb}
\usepackage{amscd}
\usepackage{amsmath}
\usepackage{amsmath}
\usepackage[pdftex]{color,graphicx}
\usepackage{setspace}
\usepackage{cancel}
\usepackage{xcolor}
\usepackage[margin=2.5cm]{geometry}
\usepackage{comment}

\newtheorem{thm}{Theorem}[section]
\newtheorem{defi}[thm]{Definition}
\newtheorem{cor}[thm]{Corollary}
\newtheorem{prop}[thm]{Proposition}
\newtheorem{lem}[thm]{Lemma}
\newtheorem{rmk}[thm]{Remark}

\begin{document}

\title{On singular problems in nonreflexive fractional Orlicz-Sobolev spaces}

	\author{Marcos L. M. Carvalho$^1$ }
	\address{$^{1}$Instituto de Matem\'atica e Estat\'istica, Universidade Federal de Goiás, Goiânia, Brasil,  74690-900.}
	\email{marcos\_leandro\_carvalho@ufg.br}
    \thanks{The first author was also partially supported by CNPq with the grant  300411/2025-1.}

	\author{Luana C. M. Lima$^2$}
	\address{$^{2}$ Departamento de Matemática, Universidade Federal de Roraima, Boa Vista, RR, Brasil, 69310-000.}
	\email{luana.lima@ufrr.br}
    \thanks{The second author was also partially supported by CAPES with the grant 88882.386249/2019-01.}

\author{Carlos A. P. Santos$^3$}
	\address{$^{3}$Departamento de Matem\'atica, Universidade de Brasília, Brasília, DF, Brasil, 70910-900.}
	\email{csantos@unb.br}
    \thanks{The third author was also partially supported by CNPq with the grant  311562/2020-5, and FAPDF under the grant 00193.00001133/2021-80.}
	
	\author{Maxwell L. Silva$^{4}$}
	\address{$^{4}$Instituto de Matem\'atica e Estat\'istica, Universidade Federal de Goiás, Goiânia, Brasil,  74690-900.}
	\email{maxwell@ufg.br}

\begin{abstract}
In this work, we deal with existence and uniqueness of positive solution $u_s$ for the singular quasilinear problem $(-\Delta_{\Phi})^su=u^{-\gamma}$ in the nonreflexive fractional Orlicz-Sobolev $ W^{s}_0L^{\Phi}(\Omega)$ for $0<s<1$. Furthermore, we show that $u_s$ converges in $L^{\Phi}(\Omega)$ to the unique positive solution $u\in W^{1}_0L^{\Phi}(\Omega)$ of the problem $-\Delta_{\Psi}u=u^{-\gamma}$ as $s \uparrow 1$, where $\Psi$ is an appropriate $N$-function equivalent to the $N$-function $\Phi$. The main difficulties to obtain existence of weak solutions for both singular quasilinear problems are that their associate energy functionals may not be well-defined on their whole natural workspaces due to the lack of the reflexivity and the presence of the singular term. To overcome these difficulties, we will use the minimization method and present a new approach to building appropriate test functions to prove that the problems have positive minimizers that we showed to be weak solutions of them, respectively. 
\end{abstract}

\keywords{fractional Orlicz-Sobolev, singular, nonreflexive, Orlicz space} 

\maketitle


\section{Introduction}

In this work, let us prove existence and uniqueness  of weak solutions for the classes of weak singular problems ($0<\gamma<1$)
\begin{equation}\label{Pa}\tag{$P_{s,\Phi}$}
\displaystyle \left\{\begin{array}{lcr} (-\Delta_\Phi)^{s}u = u^{-\gamma}, \ \  \text{in}  \ \  \Omega,  \\ 
\displaystyle u > 0, \ \ \text{in}  \ \ \ \ \ \ \ \ \ \ \ \  \Omega,  \\ 
\displaystyle u=0, \ \ \text{in}  \ \ \ \ \ \mathbb{R}^N\setminus \Omega, 
\end{array}
\right.
\end{equation} 
and 
\begin{equation}\label{Psilaplacian}\tag{$P_{\Phi}$}
\displaystyle \left\{\begin{array}{lcr} -\Delta_\Phi u = u^{-\gamma}, \ \  \text{in}  \ \  \Omega,  \\ 
\displaystyle u > 0, \ \ \text{in}  \ \ \ \ \ \ \ \ \ \ \ \  \Omega,  \\ 
\displaystyle u=0, \ \ \text{in}  \ \ \ \ \ \mathbb{R}^N\setminus \Omega, 
\end{array}
\right.
\end{equation}
where $\Omega\subset\mathbb{R}^N$ is an open bounded  domain with Lipschitz boundary, $N \geq 1$, $0<s<1$, and $\Phi$ is an $N$-function whose functional space generated may not be reflexive anymore. Beside these, we establish an appropriate $N$-function ${\Psi}$, equivalent to the $N$-function $\Phi$, in order to conclude that the problem $(P_{\Psi})$ can be seen as an $s$-limit problem of the family of problems $(P_{s,\Phi})$, as $0<s\to 1$, in some sense.


With the aim to include more information and become more realistic some mathematical models, based on partial differential equations,  sometimes makes sense to consider functional workspaces other than the classical Lebesgue and Sobolev spaces. One manner to do this is working in the Orlicz and Orlicz-Sobolev spaces. To a brief introduction here (more details can be seen in Section 2), let us consider an $N$-function $\Phi:\mathbb{R} \to \mathbb{R}$ to remind that
$$\displaystyle L^\Phi(\Omega)=\left\{u \in L^1_{loc}(\Omega); \int\limits_{\Omega}\Phi\left(\frac{|u|}{\lambda}\right)\mathrm{d}x < +\infty, \ \text{for some} \ \lambda>0 \right\}$$
is well known as the Orlicz space,  under the Luxemburg norm
\begin{equation*}
\displaystyle||u||_{\Phi} = \inf \left\{\lambda >0; \int\limits_{\Omega}\Phi \left( \frac{|u(x)|}{\lambda} \right)\mathrm{d}x \leq 1 \right\},
\end{equation*}
and
$$\displaystyle W^1L^\Phi(\Omega)=\left\{ u \in W^{1,1}_{loc}(\Omega); u~\mbox{and }|\nabla u| \in L^\Phi(\Omega) \right\}$$
is the Orlicz-Sobolev space with the norm  

\begin{equation*}
\displaystyle ||u||_{1,\Phi}=||u||_\Phi +  \sum\limits_{i=1}^{N}||\partial_i u||_\Phi,
\end{equation*}
that leads to the space $W_0^1L^\Phi(\Omega)$ defined as the weak$^*$ closure of $C_0^{\infty}(\Omega)$ in the $W^1L^\Phi(\Omega)$.

Recently, Orlicz and Orlicz-Sobolev spaces have caught  attention of a number of mathematical researchers.  The interest readers may visit, for instance, the following references R. A. Adams \cite{Adams1}, N. Fukagai, M. Ito and K. Narukawa \cite{PositiveSol}, Kufner, Alois, Oldrich John, and Svatopluk Fucik \cite{Kufner}, M. A. Krasnosel’ski, I. A. B. Rutkkii \cite{Kras} and Rao, M.M., Ren, Z.D \cite{Rao} to be introduced to this subject. On the context of solving partial differential  equations, whose natural frameworks are Orlicz and Orlicz-Sobolev spaces, we quote, for instance, the works of M. L. M. Carvalho, J. V. Gonçalves and E. D. da Silva \cite{CarvalhoGoncalvesSilva} and M. Mihailescu and V. Radulescu \cite{Radulescu1, Radulescu2}  that considered special conditions on the $N$-function $\Phi$ that lead the functional space $W^1_0 L^\Phi(\Omega)$  to be reflexive so that the use of some classical tools are still possible. 

Although the generality of problems that can be solved by using the framework of Orlicz-Sobolev spaces, we have some mathematical models that require adjustments on these workspaces to better represent the real-world situations. See, for instance, real models coming from  financial system \cite{Cont}, ultra-relativistic limits of quantum mechanics \cite{Fefferman}, semipermeable membranes and flame propagation \cite{Caffarelli}, image processing \cite{Gilboa}, and others.

In this sense, it was introduced by J. F. Bonder and A. M. Salort \cite{FBS} the  fractional version of the Orlicz Sobolev space, defined as
$$\displaystyle W^s L^\Phi (\Omega)= \left\{u \in L^\Phi(\Omega) ; \iint\limits_{\mathbb{R}^N \times \mathbb{R}^N}\Phi \left(\frac{|u(x)-u(y)|}{|x-y|^s} \right)\frac{\mathrm{d}x \mathrm{d}y}{|x-y|^N}< +\infty \right\},~0<s<1,$$
with the norm
\begin{equation*}
||u||_{s,\Phi}=||u||_\Phi +[u]_{s,\Phi},
\end{equation*}
where $$\displaystyle [u]_{s,\Phi}:=\inf \left\{\lambda >0; (1-s)\iint\limits_{\mathbb{R}^N \times \mathbb{R}^N}\Phi \left(\frac{|D_s u|}{\lambda} \right)\mathrm{d}\mu \leq 1  \right\}$$ is the well known \textbf{$(s,\Phi)$-Gagliardo seminorm}, that naturally leads to the space $W_0^sL^\Phi(\Omega)$ defined as the weak$^*$ closure of $C_0^\infty(\Omega)$ in the $W^sL^\Phi(\Omega)$. Naturally, we are considering  functions $u$ extended by $zero$  on $\mathbb{R}^N \setminus \Omega$ in the definition of $W^s L^\Phi (\Omega)$.

As a consequence of above definitions, by taking $\Phi(t)=t^p$, $t>0$, the space $W^s L^\Phi (\Omega)$ can be seen as a generalization of the classical Fractional Sobolev space 
$$\displaystyle W^{s,p}(\Omega)= \left\{u \in L^p(\Omega); \frac{u(x)-u(y)}{|x-y|^{\frac{N}{p}+s}} \in L^p(\mathbb{R}^N \times \mathbb{R}^N) \right\},$$
under the norm 
\begin{equation*}
\displaystyle ||u||_{s,p} = \left(||u||^p_p + [u]^p_{s,p} \right)^{\frac{1}{p}}  ,  
\end{equation*}
where $\displaystyle ||\cdot||_p$ is the norm of the Lebesgue space and $$\displaystyle [u]_{s,p}:=\left[ \ \iint\limits_{\mathbb{R}^N \times \mathbb{R}^N} \frac{|u(x)-u(y)|^p}{|x-y|^{N+sp}}\mathrm{d}x\mathrm{d}y \right]^{\frac{1}{p}}$$ is the \textbf{$(s,p)$- Gagliardo seminorm}. The interested reader may visit the paper of  E. Di Nezza, G. Palatucci, E. Valdinoci \cite{guia} to found rich details, properties and applications about the functional space $ W^{s,p}(\Omega)$.  

Under the framework of  Orlicz-Sobolev space it is natural to consider partial differential equations modeled by the $\Phi$-Laplacian operator 
$$\displaystyle \langle-\Delta_\Phi u,v \rangle  = \int\limits_{\mathbb{R}^N} \phi(|\nabla u|)\nabla u \nabla v \mathrm{d}x \ \mbox{for} ~u, v \in W^1 L^\Phi (\mathbb{R}^N),$$
while the fractional Orlicz-Sobolev space becomes the natural workspace for approaching problems driven by the fractional $\Phi$-Laplacian operator, defined as
\begin{eqnarray*}
    \begin{array}{lcl}
    \displaystyle (-\Delta_\Phi)^s u(x) &=& \displaystyle p.v. (1-s) \int\limits_{\mathbb{R}^N}\phi \left(\frac{|u(x)-u(y)|}{|x-y|^s} \right)\frac{u(x)-u(y)}{|x-y|^s}\frac{\mathrm{d}y}{|x-y|^N} \\
    
    &=& \displaystyle (1-s)\lim_{\varepsilon \rightarrow 0} \int\limits_{\mathbb{R}^N \backslash B_\varepsilon (0)}\phi \left(\frac{|u(x)-u(y)|}{|x-y|^s} \right)\frac{u(x)-u(y)}{|x-y|^s}\frac{\mathrm{d}y}{|x-y|^N} ,    
    \end{array}
\end{eqnarray*}
where $\displaystyle B_\varepsilon (0)$ is the  open ball centered at the origin of $\mathbb{R}^N$ with radius $\varepsilon$, and  $\phi:(0,\infty)\to (0,\infty)$ is such that 
$$\phi(t)t=\Phi'(t), \ t \in \mathbb{R},$$
that leads to its more manageable form
$$\displaystyle \langle(-\Delta_\Phi)^su,v \rangle = \frac{(1-s)}{2}\iint\limits_{\mathbb{R}^N \times \mathbb{R}^N}\phi\left(\frac{|u(x)-u(y)|}{|x-y|^s} \right)\frac{u(x)-u(y)}{|x-y|^s}\frac{v(x)-v(y)}{|x-y|^s}\frac{\mathrm{d}x \mathrm{d}y}{|x-y|^N} \  \mbox{for}~u, v \in W^sL^\Phi(\mathbb{R}^N),$$
for each $0<s<1$. Problems involving such kind of operators have been considered, for instance, in  S. Bahrouni, H. Ounaies and L. S. Tavares \cite{Sabri1}  in the context of $W^s_0 L^\Phi (\Omega)$ be still reflexive.

About problems driven by $\Phi-$Laplacian operator on non-reflexive Orlicz-Sobolev spaces, there are some results on the literature. We can quote Gossez \cite{gossez1, gossez2} who considered the homogeneous problem using Monotonicity Methods, García-Huidobro, M. et al. \cite{GarciaHuidobro} that worked eigenvalue problems using the non-smooth Lagrange multipliers technique, Mustonen and Tienari \cite{Mustonen}, Silva, Gon\c calves and Silva \cite{EGS} that used Minimization methods. Moreover, Alves, Pimenta and Silva \cite{AlvesPimentaSilva}, Alves, Bahrouni and Carvalho \cite{AlvesCarvalho}, Alves and Carvalho \cite{AlvesCarvalho2}, Silva, Carvalho, Gonçalves and Silva \cite{Silva2CarvalhoGoncalves} found critical points to the Energy functional via Mountain Pass Theorem, and Santos and Soares \cite{JeffersonMonari} treated optimal design problems using a penalization technique and a truncated minimization problems in terms of the Taylor polynomial of $N$-function $\Phi$. However, there are few works on literature approaching problems governed by the fractional $\Phi$-Laplacian operator with $\Phi$ being an $N$-function that generates its associate fractional Orlicz-Sobolev space being non-reflexive. It seems that Salort and Vivas in \cite{Salort} were the first authors to work on this direction by approaching an eigenvalue problem using the generalized Lagrange multipliers methods.

The main goal of this paper is to consider the non-homogeneous problem  $(P_{s,\Phi})$ (similar to $(P_{\Phi})$) under an $N$-function $\Phi$ that may lead its associate functional space $W^s_0L^\Phi(\Omega)$ to be non-reflexive so that its energy functional may assume infinity values on parts of the domain preventing us to approach this problem by the use of classical variational arguments. In addition, the presence of the singular term leads the energy functional to be non-differentiable even on the parts of the domain where the energy functional takes finite values. 

To overcome these difficulties, we present a new approach, inspired on ideas from Yjing \cite{Yijing}, and in a technique of constructing carefully test functions that belongs simultaneously to the part of the domain where the energy functional is finite and still permit us to show the existence of the Gateaux derivative at  minimum points of the energy functional. Besides these, by some fine estimates we are able to build an $N$-function $\Psi$ equivalent to $\Phi$ to show that the problem $(P_{\Psi})$ is an $s$-limit problem of $(P_{s,\Phi})$ as $s\to 1$, in some sense. 

Before enunciating our main results, let us remind that the $N$-function $\Phi: \mathbb{R} \rightarrow \mathbb{R}$ is of the type 
$\Phi'(t)=\phi(\vert t\vert )t, \ t \in \mathbb{R}$ where $\phi:(0,\infty)\to (0,\infty)$ is a function that, in addition, should satisfy
$$
t \mapsto t\phi(t); \quad t>0 \;\; \quad \mbox{increasing;}\leqno{(\phi_1)}
$$

$$
\lim_{t \to 0}t\phi(t)=0, \quad  \lim_{t \to +\infty}t\phi(t)=+\infty \leqno{(\phi_2)}
$$

In particular, it follows from $(\phi_1)$ and $(\phi_2)$, that
$$\Phi(t)=\int_{0}^{\vert t \vert}\phi(\tau)\tau\mathrm{d}\tau,~t \in \mathbb{R},$$
 is an N-function.

We already know that the reflexivity of Orlicz-Sobolev space $W_0^1L^\Phi(\Omega)$  and of the fractional Orlicz-Sobolev space $W_0^sL^\Phi(\mathbb{R}^N)$ are strongly connected with the property of the $N$-function $\Phi$
 to behave as
$$\Phi(2t)\leq K\Phi(t), \forall t \geq 0$$
that is well-known as the $N$-function $\Phi$ satisfying the $\Delta_2$-condition, or shortly,  $\Phi \in \Delta_2$. More specifically, the reflexivity of Orlicz-Sobolev space $W_0^1L^\Phi(\Omega)$  and of the fractional Orlicz-Sobolev space $W_0^sL^\Phi(\mathbb{R}^N)$ are equivalent to the fact the the $N$-function $\Phi$ and its \textbf{conjugate N-function}  $$\displaystyle \tilde{\Phi}(s) = \sup\{s t-\Phi(t); t>0 \}$$
should satisfy the $\Delta_2$-condition, or shortly,  $\Phi, \displaystyle \tilde{\Phi} \in \Delta_2$ (see Proposition \ref{reflexividade} below). So, our main interest in this paper is when at least one of the $N$-function  $\Phi$  and $\displaystyle \tilde{\Phi}$ does not satisfy the $\Delta_2$-condition.

To complete our overview to state our main results, just remains to make clear the meaning of weak solutions throughout this paper. First, we set the weak solution for the problem \eqref{Pa}.
\begin{defi}
		A function $u=u_s\in W^s_0 L^\Phi(\Omega)$ is said to be a \textbf{weak solution of \eqref{Pa}} if 
		$$\frac{(1-s)}{2}\iint\limits_{\mathbb{R}^N \times \mathbb{R}^N}\phi\left(\frac{|u(x)-u(y)|}{|x-y|^s} \right)\frac{u(x)-u(y)}{|x-y|^s}\frac{v(x)-v(y)}{|x-y|^s}\frac{\mathrm{d}x \mathrm{d}y}{|x-y|^N}=\int\limits_{\Omega}u^{-\gamma}vdx,~\mbox{for all}~v\in W^s_0 L^\Phi(\Omega).$$
\end{defi}
Finally, a weak solution to the problem \eqref{Psilaplacian} is understood as below.
\begin{defi}
		A function $u\in W^1_0 L^\Phi(\Omega)$ is said to be a \textbf{weak solution of \eqref{Psilaplacian}} if 
		$$\int\limits_{\Omega} \phi(|\nabla u|)\nabla u \nabla v \mathrm{d}x =\int\limits_{\Omega}u^{-\gamma}vdx,~\mbox{for all}~v\in W^1_0 L^\Phi(\Omega).$$
\end{defi}

Now, we are ready to state our first result. 

\begin{thm}\label{pricipal-1}
Assume that $0<s, \gamma<1$, and $(\phi_1)-(\phi_2)$ hold. Then Problem \eqref{Pa} admits a unique positive weak solution $u:=u_s \in W_0^s L^\Phi (\Omega)$. 
\end{thm}

As a consequence of the above theorem, we establish a unique solution for the problem
\begin{equation*}
\displaystyle \left\{\begin{array}{lcr} \displaystyle  \ p.v. (1-s)\int\limits_{\mathbb{R}^N} (u(x)-u(y))e^{\frac{|u(x)-u(y)|^2}{|x-y|^{2s}}}\frac{\mathrm{d}y}{|x-y|^{N+s}}=u(x)^{-\gamma} \ \  \text{in}  \ \  \Omega;  \\ 
\displaystyle u > 0, \ \ \ \ \ \ \ \ \ \ \ \ \ \ \ \ \ \ \ \ \ \ \ \ \ \ \ \ \ \ \ \ \ \ \ \ \ \ \ \ \ \ \ \ \ \ \ \ \ \ \ \ \ \ \ \ \ \ \ \ \ \ \ \ \ \ \ \   \text{in}   \ \  \Omega;  \\ 
\displaystyle u=0, \ \ \ \ \ \ \ \ \ \ \ \ \ \ \ \ \ \ \ \ \ \ \ \ \ \ \ \ \ \ \ \ \ \ \ \ \ \ \ \ \ \ \ \ \ \ \ \ \ \ \ \ \ \ \ \ \ \  \text{in}  \ \ \ \ \ \mathbb{R}^N\setminus \Omega. 
\end{array}
\right.
\end{equation*}
in the context of the non-reflexive space  $W_0^s L^\Phi (\Omega)$, because its associate $N$-function 
$$\Phi(t)=\frac{{e^{t^2}-1}}{2},~t>0$$
satisfies $(\phi_1)-(\phi_2)$; however,
$$\ell := \inf_{t \geq 0} \frac{t^2 \phi(t)}{\Phi(t)} =2 ~~\mbox{and}~~m:= \sup_{t \geq 0} \frac{t^2 \phi(t)}{\Phi(t)}= \infty,$$
which implies that $\tilde{\Phi}\in \Delta_2$ but $\Phi \not\in \Delta_2$. See Proposition \ref{reflexividade}. 

By taking advantage of the closeness in the arguments, we also prove the below theorem.

\begin{thm}\label{pricipal-2}
Assume that $0< \gamma<1$, and $(\phi_1)-(\phi_2)$ hold. Then Problem \eqref{Psilaplacian} admits a unique positive weak solution $u \in W_0^1L^{\Phi}(\Omega)$. 
\end{thm}

To state our last result, let us point out that under the change of variable  $\rho=tr^{1-s}$,  we have well defined the limit 
\begin{equation}\label{Psi1}
\Psi(t)=\lim_{s \uparrow 1}(1-s)\int_0^1 \int_{\mathbb{S}^{N-1}} \Phi(t|z_N|r^{1-s})\mathrm{d}S_z \frac{\mathrm{d}r}{r}=\int_0^t \int_{\mathbb{S}^{N-1}} \Phi(\rho|z_N|)\mathrm{d}S_z \frac{\mathrm{d}\rho}{\rho},~t\in \mathbb{R},
\end{equation} 
where $z_N$ is the last coordinate of the vector $z=(z_1,...,z_N) \in \mathbb{R}^N$ and $\mathbb{S}^{N-1}$ is the unitary sphere centered at the origin of $\mathbb{R}^N$. We highlight that $\Psi$ is a strictly convex $N$-function  equivalent to the $\Phi$ that satisfies 
$$\Psi(t)\leq |\mathbb{S}^{N-1}|\int_0^t\frac{\Phi(\rho)}{\rho}d\rho,~t\in \mathbb{R}.$$
See section 5 for the proofs and details.

As a consequence of the last equality in \eqref{Psi1}, we have:
\begin{enumerate}
    \item[1)] $\Phi(t)=|t|^p, t \in \mathbb{R}^+$ leads to 
   $$\displaystyle \Psi(t) = \frac{k_{N,p}}{p} |t|^p,~t\in \mathbb{R},~\mbox{where }k_{N,p}:=\int\limits_{\mathbb{S}^{N-1}}|z_N|^p \mathrm{d}S_z$$ 
   whose associate operator is
   $$-\Delta_\Psi u=-k_{N,p}\mbox{div}(|\nabla u|^{p-2}\nabla u);$$
   
   \item[2)] for $\Phi(t)=t^p|log \ t|, t \in \mathbb{R}^+$ and $p>1$, then  
   $$\Psi(t)=\frac{t^p}{p}\left(k_{N,p}|log t|+k_{log,N,p}+\frac{k_{N,p}}{p}\right),~\mbox{where }k_{log,N,p}:=\int\limits_{\mathbb{S}^{N-1}}|z_N|^p|log|z_N||\mathrm{d}S_z,$$
 and 
   $$-\Delta_\Psi u =-k_{N,p}\mbox{div}(|\nabla u|^{p-2}|log|\nabla u||\nabla u)-(pk_{log,N,p}+2k_{N,p})\mbox{div}(|\nabla u|^{p-2}\nabla u);$$
   
   \item[3)] if $1<q<p$ and $\Phi(t)=\max\{t^p, t^q \}$, then we have
   \begin{equation*}
    \displaystyle 
    \Psi(t)=\left\{\begin{array}{lcr}\displaystyle \frac{k_{N,q}}{q} t^q, \ \ \ \ \  \ \ \ \ \ \ \ \ \ \  \ \ \ \ \  \ \ \ \ \  \ \ \ \ \  \ \ \ \ \  \ \ \ \ \ \text{if} \ t \leq 1;  \\ 
    \displaystyle \frac{t^q}{q}\int\limits_{|z_N|\leq \frac{1}{t}}|z_N|^q \mathrm{d}S_z + \frac{t^p}{p}\int\limits_{|z_N|>\frac{1}{t}}|z_N|^p \mathrm{d}S_z \\ 
    \displaystyle + \left(\frac{1}{q} - \frac{1}{p} \right) \int\limits_{|z_N|> \frac{1}{t}}\mathrm{d}S_z\ \ \ \ \  \ \ \ \ \  \ \ \ \ \  \ \ \ \ \ \  \text{if} \ t>1. 
        \end{array}
        \right.
    \end{equation*} 
 and
    \begin{equation*}
    \displaystyle 
    -\Delta_\Psi u=\left\{
    \begin{array}{lcr}
    \displaystyle -k_{N,q}\mbox{div}(|\nabla u|^{q-2}\nabla u),&&  \text{if} \ |\nabla u| \leq 1,  \\ 
    \displaystyle -\mbox{div}(\int\limits_{|z_N|\leq \frac{1}{|\nabla u|}}|z_N|^q \mathrm{d}S_z|\nabla u|^{q-2}\nabla u+\int\limits_{|z_N|> \frac{1}{|\nabla u|}}|z_N|^p \mathrm{d}S_z|\nabla u|^{p-2}\nabla u),&&\text{if} \ |\nabla u|>1;
        \end{array}
    \right.
    \end{equation*}
    
   \item[4)]  for $1<p<q$ and $\Phi(t)=|t|^p + |t|^q, t \in \mathbb{R}^+$, we conclude that
   $$\displaystyle \Psi(t) = \frac{k_{N,p}}{p} |t|^p+\frac{k_{N,q}}{q} |t|^q,~t\in \mathbb{R},$$
   where $k_{N,p}$ and $k_{N,q}$ were defined in $1)$, and
   $$-\Delta_\Psi u =-k_{N,p}\mbox{div}(|\nabla u|^{p-2}\nabla u)-k_{N,q}\mbox{div}(|\nabla u|^{q-2}\nabla u).$$
\end{enumerate}

So, we are ready to present a result that contributes in the understanding of the behavior of Problem \eqref{Pa} as $s \uparrow 1$. This kind of subject was considered, for instance, by Bonder and Salort in \cite{fractionalplaplace1} and Bonder, Salort, and Vivas in \cite{fractionalplaplace3} for limiting situations of the fractional $p$-Laplacian operator as $s \uparrow 1$, while Bonder, Silva, and Spedaletti in \cite{fractionalplaplace2}  presented a  behavior of 
eigenvalues problems to the fractional p-Laplacian operator as $s \uparrow 1$, by using the $\Gamma$-convergence method.

\begin{thm}\label{gammaconvergence}
Assume $0<s, \gamma<1$, $(\phi_1)-(\phi_2)$, and {$\Phi\in \Delta_2$}.
Let $u_s \in W_0^s L^\Phi(\Omega)$ be the unique weak solution to the problem \eqref{Pa}, for $0<s<1$, given by Theorem \ref{pricipal-1}. Then $\displaystyle u_s\to u$ in $L^\Phi(\Omega)$, when $s\ \uparrow 1$, where $u\in W_0^1L^\Phi(\Omega)$ is the unique weak solution of the problem $(P_{\Psi})$, where $\Psi$ was defined in \eqref{Psi1}.
\end{thm}

The study of singular problems of the type ($P_{\Phi}$) is not new under the assumption that $\Phi$ satisfies the $\Delta_2$-condition. See, for instance, Carvalho, M. L., Goncalves, J. V., Silva, E. D., and Santos, C. A. P. \cite{CarvalhoGoncalvesSilvaSantos}, where the authors proved  existence and uniqueness of solutions for a quasilinear elliptical problem that may be singular at the origin, and also proved a comparison principle. However, the main novelty in Theorem \ref{pricipal-2} is the fact that the $N$-function $\Phi$ may not satisfy the $\Delta_2$-condition so that the arguments to overcome this difficulty need to be news and other than the classical ones used when $\Phi$ satisfies the $\Delta_2$-condition. It seems that this work is the first one to approach the difficulties that come from both the singularity and the lack of the $\Delta_2$-condition.

In the sequel, let us highlight the main contributions to the literature of this paper:
\begin{enumerate}
    \item[1)] the singularity in the problems ($P_{s,\Phi}$) and ($P_{\Phi}$) may be twofold in the sense that one comes from the presence of the term $u^{-\gamma}$ with $\gamma>0$ and the other one comes from the lack of the $\Delta_2$-condition that leads the energy functional be not finite in some parts of the domain, 
 \item[2)]  despite of the lack of the regularity of the associated energy functional and the restrictions required to establish its domain, we are still able to apply Variational method, 
    
    \item[3)] we carefully construct test functions that belong to the effective domain $D_{\Phi}$ (see definition in Section 3), to demonstrate the validity of the weak formulation of solutions for the problems, 
    
    \item[4)] with the help of the above building of test functions, we establish a new approach to prove that a local minimum in $D_\Phi \cap D_{\tilde{\Phi}}$ (see definitions of $D_\Phi$ and $D_{\tilde{\Phi}}$ later) is a weak solution for problems involving operator and singular non-linearities, when the effective domain is a convex and proper subset.

    \item[5)] Theorem \ref{pricipal-2} as well as Theorem \ref{pricipal-1} are new in the context of obtaining solutions for non-local \eqref{Pa} and local \eqref{Psilaplacian}  problems when the fractional Orlicz-Sobolev and Orlicz-Sobolev spaces are not reflexive, while Theorem \ref{gammaconvergence}  is also new for the same reason and because it involves the singular term.
\end{enumerate}

This work is structured as follows: Initially we present a historical context, motivations of our work and also the statement of our main results. In Section 2, we present some concepts and properties to ease the reading of the readers. In the subsequent section, we prove the first two theorems of this work through a sequence of lemmas and propositions, and in the last section we prove Theorem \ref{gammaconvergence}, based on ideas from \cite{FBS}.

\section{About the workspace}

To ease the reading of the readers, we are going to present on this section the main definitions and properties, related to the our functional workspaces. Beside these, let us state and prove two results that will be very useful in our approach. We begin with the Orlicz Spaces.

\subsection{Orlicz spaces}

In this work consider $\Omega \subset \mathbb{R}^N$ be a bounded domain. Also consider $\phi: (0,+\infty) \rightarrow (0,+\infty)$ be a function satisfying the hypotheses $(\phi_1)$ and $(\phi_2)$. So, it is well known that  
 $$\displaystyle\Phi(t):=\int_{0}^{|t|}\phi(s)s \mathrm{d}s,~t \in \mathbb{R},$$
defines an $N$-function. To it, the \textbf{conjugate N-function}  is defined by 
$$\displaystyle \widetilde{\Phi}(s) = \sup\{st-\Phi(t); t>0 \},~ s \in \mathbb{R}$$
that imply { $\widetilde{\widetilde{\Phi}} = \Phi$ } and the \textbf{Young's inequality} 
$$st \leq \Phi(t)+ \widetilde{\Phi}(s),$$ where the equality holds  (\textbf{Young's equality})  if, and only if, $s=\tilde{\phi}(t)t:=\widetilde{\Phi}'(t)$  or $t=\phi(s)s$ with $\tilde{\phi}$ satisfying 
$$
   \tilde{\Phi}(t)=\int_0^{\vert t \vert} \tilde{\phi}(\tau)\tau\mathrm{d}\tau,~t \in \mathbb{R}. 
$$

An N-function $\Phi$ satisfies the \textbf{$\Delta_2$-condition} if for $K>1$ we have 
\begin{equation}\label{delta0}
\Phi(2t)\leq K\Phi(t), \forall t \geq 0,
\end{equation}
and, in particular, if \eqref{delta0} is satisfied to all $t\geq t_0$, for some $t_0>0$, we say that $\Phi$ satisfies the \textbf{$\Delta_2(\infty)$-condition at infinity}, in short, $\Phi \in \Delta_2(\infty)$. 

\begin{prop}[See \cite{Rao}]\label{thm21}
Let $\Phi$ and $\tilde{\Phi}$ be conjugated N-functions. Then the following statements are equivalents:
\begin{itemize}
    \item[(i)] $\Phi \in \Delta_2(\infty)$;
    \item[(ii)] there exist $1<m<+\infty$ and $t_0>0$ such that $$\displaystyle \frac{t^2\phi(t)}{\Phi(t)}\leq m, \ t \geq t_0;$$
    \item[(iii)] there exist $1<\widetilde{\ell}<+\infty$ and $t_0 \geq 0$ such that $$\displaystyle \widetilde{\ell} \leq \frac{t^2\tilde{\phi}(t)}{\tilde{\Phi}(t)}, \ t \ge t_0.$$
\end{itemize}
\end{prop}





We define the \textbf{Orlicz class} of the N-function $\Phi$ as the following convex set
$$\displaystyle \mathcal{L}_\Phi(\Omega)=\left\{u:\Omega \longrightarrow \mathbb{R} \ \text{mensurable}; \int\limits_\Omega \Phi(u(x))\mathrm{d}x <+\infty \right\}$$
that is a  vector space if and only if $\Phi \in \Delta_2(\infty)$, while the \textbf{Orlicz space} is defined by 
$$L^{\Phi}(\Omega)=\left\{u: \Omega \longrightarrow \mathbb{R} \ \ \text{mensurable}; \int\limits_{\Omega}\Phi\left(\frac{|u|}{\lambda}\right)\mathrm{d}x<+\infty,~\mbox{for some }\lambda>0 \right\}$$
that is a Banach space, under the \textbf{Luxemburg norm} 
\begin{equation*}\label{normaorlicz}
\displaystyle||u||_{\Phi} = \inf \left\{\lambda >0; \int\limits_{\Omega}\Phi \left( \frac{|u(x)|}{\lambda} \right)\mathrm{d}x \leq 1 \right\},
\end{equation*}
continuously embedded into $L^1(\Omega)$, see \cite{Kufner}. Orlicz space $L^{\Phi}(\Omega)$ is the smallest vector
space containing $\mathcal{L}_{\Phi}(\Omega)$, that is,
$\mathcal{L}_\Phi(\Omega) \subseteq L^\Phi(\Omega)$
with the equality holding if and only if $\Phi \in \Delta_2$. Furthermore, the space $E^{\Phi}(\Omega):=\overline{C_0^{\infty}(\Omega)}^{L^{\Phi}(\Omega)}$ is such that 
$E^{\Phi}(\Omega)=L^{\Phi}(\Omega)$ if and only if $\Phi \in \Delta_2(\infty).$
Summing up, we have $$E^\Phi(\Omega) \subset \mathcal{L}_\Phi(\Omega) \subseteq L^\Phi(\Omega) \subseteq L^1(\Omega)$$
with the three equalities occurring simultaneously
 if and only if $\Phi \in \Delta_2(\infty)$. 
 
An important consequence of Young’s inequality is \textbf{Hölder inequality} $$\displaystyle \left|\int\limits_{\Omega}uv \mathrm{d}x \right| \leq 2||u||_{\Phi}||v||_{\widetilde{\Phi}}, \forall u \in L^{\Phi}(\Omega), \forall v \in L^{\widetilde{\Phi}}(\Omega).$$

\begin{prop}\label{LPhicontinuousL1}(See \cite{Kufner})
Assume that $\Omega \subset \mathbb{R}^N$ has finite Lebesgue messure. Then  Orlicz Space $L^\Phi(\Omega)$ is continuous embedding into $L^1(\Omega)$.
\end{prop}

More important facts on $L^\Phi(\Omega)$.
\begin{prop}\label{reflexividade}(See \cite{Kras})
\begin{enumerate}
\item[($i$)] $E^\Phi(\Omega)$ is a closed subspace of $L^\Phi(\Omega)$ (hence a Banach space itself). More, $E^\Phi(\Omega)$ is
separable. 
    \item[$ii$)] Orlicz space $L^\Phi(\Omega)$ is reflexive if and only if $\Phi, \tilde{\Phi} \in \Delta_2(\infty)$,
     \item[$iii$)] $L^\Phi(\Omega)$ is separable if $\Phi \in \Delta_2(\infty)$.
\end{enumerate}
\end{prop}

Using the same ideas of \cite{Salort} and \cite{FN2}, we can prove the following lemma.
\begin{lem}
Assume $(\phi_1)$ and $(\phi_2)$ hold. Then
$$\displaystyle ||u||_\Phi \to +\infty\qquad\Rightarrow\qquad \int\limits_\Omega \Phi(|u|)\mathrm{d}x\to +\infty , \ \ \ \forall u \in W^s L^\Phi(\Omega).$$
If, in addition, $\Phi\in \Delta_2$ holds, then 
$$\min\{||u||_{\Phi}^{\ell}, ||u||_{\Phi}^{m} \} \leq \int\limits_\Omega \Phi(|u|)\mathrm{d}x \leq \max\{||u||_{\Phi}^{\ell}, ||u||_{\Phi}^{m} \}, \ \forall u \in L^\Phi(\Omega).$$   
\end{lem}


Below, let us recall Orlicz-Sobolev
space.

\subsection{Orlicz-Sobolev space}
As we know, \textbf{Orlicz-Sobolev space} generated by the $N$-function $\Phi$ is defined by 
$$\displaystyle W^1L^\Phi(\mathbb{R}^N)=\left\{ u \in W^{1,1}_{loc}(\mathbb{R}^N); u~\mbox{and }|\nabla u| \in L^\Phi(\mathbb{R}^N) \right\}$$
with the norm
\begin{equation*}
\displaystyle ||u||_{1,\Phi}=||u||_\Phi +  \sum\limits_{i=1}^{N}||\partial_i u||_\Phi,
\end{equation*}
and $W_0^1L^\Phi(\Omega)$ as the weak$^*$ closure of $C_0^\infty(\Omega)$ in $W^1L^\Phi(\Omega)$. In this space, the modular
Poincar\'e’s Inequality
$$\int_\Omega\Phi(|u|)dx\leq \int_\Omega\Phi(d|\nabla u|)dx,~\forall u\in W_0^1L^\Phi(\Omega),$$
holds, where $d=\mbox{diam}(\Omega)$, and consequently 
$$\|u\|_\Phi\leq 2d\|\nabla u\|_\Phi,~\forall u\in W_0^1L^\Phi(\Omega),$$
which implies that the functional $\|\cdot\|:=\|\nabla \cdot\|$ defines an equivalent norm on $W_0^1L^\Phi(\Omega)$.

Some important facts.
\begin{prop}[See \cite{Adams2}]
We have the following:
\begin{enumerate}
    \item[$i)$] Orlicz-Sobolev space $W_0^1L^\Phi(\Omega)$ is reflexive if and only if $\Phi, \tilde{\Phi} \in \Delta_2(\infty)$,
\item[$ii)$] $W^1_0L^\Phi(\Omega)=\overline{C_0^\infty(\Omega)}^{\|\|_{1,\Phi}}$ if $\Phi \in \Delta_2(\infty)$,
      \item[$iii)$] $W_0^1L^\Phi(\Omega)$ is separable if $\Phi \in \Delta_2(\infty)$.
\end{enumerate}
\end{prop}
Below, let us recall some facts related to Fractional Orlicz-Sobolev spaces.

\subsection{Fractional Orlicz-Sobolev spaces}

As we already mentioned, the \textbf{fractional Orlicz-Sobolev space}  
 was introduced by J. F. Bonder and A. M. Salort \cite{FBS} as
$$
\displaystyle W^s L^\Phi(\Omega):=\left\{u \in L^{\Phi}(\Omega); D_s u \in L^\Phi(\mathbb{R}^N \times \mathbb{R}^N, \mathrm{d}\mu )\right\},
$$
under the norm
\begin{equation}\label{normafractional}
||u||_{s,\Phi}=||u||_\Phi +[u]_{s,\Phi},
\end{equation}
where  $u$ is understood as being $zero$ on $\mathbb{R}^N \setminus \Omega$, 
$$\displaystyle [u]_{s,\Phi}:=\inf \left\{\lambda >0; (1-s)\iint\limits_{\mathbb{R}^N \times \mathbb{R}^N}\Phi \left(\frac{|D_s u|}{\lambda} \right)\mathrm{d}\mu \leq 1  \right\}$$ is the \textbf{$(s,\Phi)$-Gagliardo seminorm,} and
\begin{center}
$\displaystyle D_s u =\frac{u(x)-u(y)}{|x-y|^s}, \ \ \ \   \mathrm{d}\mu:=\frac{\mathrm{d}x\mathrm{d}y}{|x-y|^N}, \ \forall (x,y) \in (\mathbb{R}^N \times \mathbb{R}^N) $.  
\end{center}

A first important property.
\begin{prop}\label{teo2.5}(See \cite{ABS})
$C_0^{2}(\Omega) \subset W^s L^\Phi(\Omega).$
\end{prop}

As a consequence, we have well-defined the Banach space 
$W_0^s E^\Phi(\Omega)$ as the closure of $C_0^\infty(\Omega)$ in the 
norm \eqref{normafractional}.

\begin{prop}
Let $\Omega$ be an open bounded subset of $\mathbb{R}^N$ with Lispchitz boundary and $\Phi$ an N-function. Then:
\begin{enumerate}
    \item[$i)$] $W^s L^\Phi(\Omega)$ is a Banach space with the norm \eqref{normafractional}, $W^s_0E^\Phi(\Omega)$ is a closed subspace of $W^s L^\Phi(\Omega)$ (hence a Banach space it self). Furthermore, $W^s_0E^\Phi(\Omega)$ is separable, 
     \item[$ii)$]\label{impontant} the spaces $W^s L^\Phi(\Omega)$ and $W^sE^\Phi(\Omega)$ are isometrically identified with $L^\Phi(\Omega) \times
L^\Phi(\mathbb{R}^N \times \mathbb{R}^N; d\mu)$ and $E^\Phi(\Omega)\times E^\Phi(\mathbb{R}^N \times \mathbb{R}^N; d\mu)$, respectively.
\end{enumerate}
\end{prop}

As a consequence of the last item $ii)$ and the fact $$(E^{\widetilde{\Phi}})'=L^\Phi \ \text{and} \ (E^\Phi)'=L^{\widetilde{\Phi}}, $$
we have that the space $W^s L^\Phi(\Omega)$ is a closed subspace of $L^\Phi(\Omega) \times L^\Phi(\mathbb{R}^N \times \mathbb{R}^N; d\mu)$ being this space the dual of the following separable space $(E^{\widetilde{\Phi}} \times E^{\widetilde{\Phi}}(d\mu))'$, which implies by Banach-Alaoglu theorem that $W^s L^\Phi(\Omega)$ is  weak* closed in $L^\Phi(\Omega) \times 
L^\Phi(d\mu)$. Thus, by denoting $W_0^s L^\Phi(\Omega)$ as the weak* closure of $C_{0}^{\infty}(\Omega)$ in $W^s L^\Phi(\Omega)$, we have that $W_0^s L^\Phi(\Omega)$
is a weak* closed subset of the dual of the separable space $(E^{\widetilde{\Phi}}(\Omega))' \times (E^{\widetilde{\Phi}}(\mathbb{R}^N \times \mathbb{R}^N; d\mu))'$, because 
$$W_0^s L^\Phi(\Omega)\subset L^\Phi(\Omega) \times L^\Phi(\mathbb{R}^N \times \mathbb{R}^N; d\mu)=(E^{\widetilde{\Phi}}(\Omega))' \times (E^{\widetilde{\Phi}}(\mathbb{R}^N \times \mathbb{R}^N; d\mu))'.$$

We recall that $W^s_0L^\Phi(\Omega)$ may not be reflexive so that the  weak$^*$ convergence is the our main manageable tool. The next result is a consequence of the above information.

\begin{lem} \label{Estrela} If $(u_n) \subset W_0^s L^\Phi(\Omega) $ is a bounded sequence, then there are a subsequence of $(u_n)$, still denoted by itself, and $u \in  W_0^s L^\Phi(\Omega)$ such that
	$$
	u_n \stackrel{*}{\rightharpoonup} u \quad \mbox{in} \quad W_0^s L^\Phi(\Omega).
	$$
\end{lem}


Below, let us present some important definitions and results for our approach, besides additional assumptions that are need  for compact-embedding results.

\begin{defi}
Given N-functions $\Phi$ and $\Psi$, we say that \textbf{$\Phi$ grows essentially more slowly (near infinity) than $\Psi$} if
$$\displaystyle \lim_{t \rightarrow \infty}\frac{\Phi(t)}{\Psi(\beta t)}=0,$$
for all $\beta > 0$. In particular, for any $c>0$ there exists $T>0$ such that $$\Phi(t) \leq \Psi(ct), \ \text{for} \ t \geq T.$$
\end{defi}

By letting an $N$-function $\Phi$  such that 
\begin{equation}\label{2.12}
\displaystyle \int_{1}^{\infty}\left(\frac{t}{\Phi(t)} \right)^{\frac{s}{N-s}}\mathrm{d}t=\infty ~\mbox{and }\int_{0}^{1}\left(\frac{t}{\Phi(t)} \right)^{\frac{s}{N-s}}\mathrm{d}t < \infty
\end{equation}
hold, we can define the critical function $\Phi_*(t) := \Phi(H^{-1}(t))$ by $$\displaystyle H(t)=\left(\int_{0}^{t}\left(\frac{\tau}{\Phi(\tau)} \right)^{\frac{s}{N-s}}\mathrm{d}\tau \right)^{\frac{N-s}{N}}$$
and state the below result, whose proof can be found at \cite[Theorem. 9.1]{Alberico}. 

\begin{prop}\label{compact}
Let $\Omega \subset \mathbb{R}^N$ be a bounded Lipschitz domain and $\Phi$ an N-function satisfying \eqref{2.12}. If the N-function
$\Psi$ grows essentially more slowly than $\Phi_*$ near to infinity, then the embedding $W_0^s L^\Phi(\Omega) \subset L^{\Psi} (\Omega)$
is compact.
\end{prop}

In particular, we obtain from the previous theorem, whose proof is in \cite[Proposition 2.3]{Salort}, the next one.

\begin{prop}\label{211}
Under the hypothesis of the previous theorem, the embedding of $W_0^s L^\Phi(\Omega)$ into $L^\Phi(\Omega)$ is compact. In particular,  $W_0^sL^\Phi(\Omega)$ is continuously embedded into $L^1(\Omega)$.
\end{prop}

To state a version of the well-known \textbf{Poincaré's Inequality} for $W_0^s L^\Phi(\Omega)$, let us  denote the \textbf{modulars} on $L^\Phi(\Omega)$ and $W_0^s L^\Phi(\Omega
)$, respectively, by
$$\rho_{\Phi}(u)=\int\limits_{\Omega}\Phi(|u|)\mathrm{d}x, \ \mbox{and} \ I_1(u)=(1-s)\iint\limits_{\mathbb{R}^N \times \mathbb{R}^N}\Phi(|D_s u|)\mathrm{d}\mu.$$ 

\begin{prop}(See \cite{FBS,Salort})\label{Poincare} 
Assume that $\Omega$ is a bounded domain and $\Phi$ satisfies $(\phi_1)-(\phi_2)$. Then, there exists a constant $C>0$, depending on $s,N, \Phi$ and $\Omega$, such that
$$
    \displaystyle \rho_\Phi(u) \leq \iint\limits_{\mathbb{R}^N \times \mathbb{R}^N}\Phi(C|D_s u|)\mathrm{d}\mu.
$$
for every $ 0< s < 1$ and $u \in L^\Phi(\Omega)$. In addition, if $\Phi$ satisfies the $\Delta_2$-condition then
$$
    \displaystyle \rho_\Phi(u) \leq  C I_1(u),
$$
where $C$ is not depend of $s$.
\end{prop}

As a consequence of Poincaré's Inequality, we have that the $(s,\Phi)$-Gagliardo seminorm $[\cdot]_ {s,\Phi}$ is an equivalent norm on $W_0^s L^\Phi(\Omega)$ to that one defined in \eqref{normafractional} so that the fractional Orlicz-Sobolev space $W_0^s L^\Phi(\Omega)$ will be equipped with $(s,\Phi)$-Gagliardo seminorm $[\cdot]_ {s,\Phi}$ throughout this paper.

About the functional $I_1$.
\begin{lem}[See Propositon 3.1 and 3.2 of \cite{Salort}]\label{I1-coer}
Let $\Phi$ be an N-function. Then $I_1$ is coercive and weak{*} lower semicontinuous.
\end{lem}





The proof of below result can be found in \cite{Salort} and \cite{Sabri1}.

\begin{prop}\label{prop2.3}(See \cite{Sabri1,Salort})
Assume $(\phi_1)$ and $(\phi_2)$ hold. Then:
$$[u]_{s,\Phi}\to +\infty\qquad\Rightarrow\qquad  I_1(u)\to+\infty,~~~~u \in W_0^s L^\Phi(\Omega).$$
If, in addition, $\Phi\in \Delta_2$ holds, then
\begin{equation*}\label{i}
\min\{{[u]^{\ell}_{s,\Phi}, [u]_{s,\Phi}^{m}}\} \leq I_1(u) \leq \max\{ [u]^{\ell}_{s,\Phi},[u]_{s,\Phi}^{m}\}, \ \forall u \in W_0^sL^\Phi(\Omega).
\end{equation*}
\end{prop}

{

\subsection{Auxiliary Lemmas} 

Below, let us state and prove two important results in our strategy of proving Theorem \ref{pricipal-1}. The first result is important to prove that the unique weak solution of \eqref{Pa} is positive, while the second one generalizes a classical result for Sobolev spaces to our workspace.

\begin{lem}\label{Phi}
Let $\Phi$ be an N-function, and $a,b \geq 0$ with $a\not=b$. Then:
\begin{enumerate}
    \item[(i)] one has 
    $$\displaystyle {\Phi(a)-\Phi(b)} \leq \big(\phi(a)a +\phi(b)b\big)({a-b}).$$
    
    \item[(ii)] one has 
    $$\phi\left((1-t )a + tb \right)\left((1-t )a + tb \right)\leq \phi(a)a + \phi(b)b, \ \forall\  0\leq t\leq1.$$
\end{enumerate}
\end{lem}

\begin{proof} Spite of the proofs follow from classical arguments, let us present them for the convenience
 of the readers. 
\begin{enumerate}
%
%
%
\item[(i)] As above and using  that $s\mapsto s\phi(s)$ is increasing  
\begin{eqnarray*}
\begin{array}{lll}
    \displaystyle\frac{\Phi(a)-\Phi(b)}{a-b} & \leq &   \phi((1-\theta)a +\theta b)[(1-\theta)a+\theta b] \\\\
     & \leq & \displaystyle \phi((1-\theta)\max\{a,b\} +\theta \max\{a,b\})[(1-\theta)\max\{a,b\}+\theta \max\{a,b\}]\\ \\
     &= & \displaystyle \phi(\max\{a,b\})\max\{a,b\} \leq \phi(a)a+\phi(b)b,
\end{array}
\end{eqnarray*}
for some $\theta \in (0,1)$.
\item[(ii)] Here, we can use the same ideas of item i).
\end{enumerate} This ends the proof.
\end{proof}

The next result guarantee that the positive and negative parts of $u \in W_0^sL^\Phi(\Omega)$ belongs to $W_0^s L^\Phi(\Omega).$

\begin{lem}\label{u+}
Given $u \in W_0^sL^\Phi(\Omega)$, then $u^+,u^- \in W_0^sL^\Phi(\Omega),$ where

\begin{equation*}
\displaystyle \left\{\begin{array}{lcr}   \displaystyle u^+=\max\{u,0 \};\\ 
\displaystyle u^-=\max\{-u,0 \};   \\ 
\displaystyle u=u^+ - u^-.
\end{array}
\right.
\end{equation*}

\end{lem}

\begin{proof}
Given $u \in W_0^sL^\Phi(\Omega)$, we have 
$$\displaystyle\iint\limits_{\mathbb{R}^N \times \mathbb{R}^N}\Phi\left(\frac{|D_s u|}{\lambda}\right)\mathrm{d}\mu<\infty~\mbox{for some } \lambda>0,$$ 
which implies by the facts 
$$|D_su^+|=\frac{|u^+(y)-u^+(x)|}{|x-y|^s}\leq |D_su|$$
and $\Phi$ be increasing that
$$\displaystyle \iint\limits_{\mathbb{R}^N\times \mathbb{R}^N}\Phi\left(\frac{|D_s u^+|}{\lambda}\right)\mathrm{d}\mu \leq \iint\limits_{\mathbb{R}^N \times \mathbb{R}^N}\Phi\left(\frac{|D_s u|}{\lambda}\right)\mathrm{d}\mu\;\;<\;\;\infty.$$
Hence, $u^+ \in W_0^s L^\Phi (\Omega).$ To conclude that $u^+\in W_0^sL^\Phi(\Omega)$, it remains to justify that there exists a $(v_n) \subset C^{\infty}_0(\Omega)$ such that $v_n\stackrel{*}{\rightharpoonup}u^+$. In fact, since $u\in W_0^sL^\Phi(\Phi)$, exists a $(u_n)\subset W_0^sL^\Phi(\Omega)$ such that $u_n\stackrel{*}{\rightharpoonup}u$. So, due to the dominated convergence theorem, we have  
$$\int\limits_\Omega u_n^+v dx\to \int\limits_\Omega u^+v dx\qquad\mbox{and}\qquad\iint\limits_{\mathbb{R^N}\times\mathbb{R^N}}\phi(|D_s u_n^+|)D_s u_n^+D_s vd\mu\to \iint\limits_{\mathbb{R^N}\times\mathbb{R^N}}\phi(|D_s u^+|)D_s u^+D_s vd\mu$$
for all $v\in E^{\widetilde{\Phi}}(\Omega)$. The Proof for $u^-$ is similar. This ends the proof.
\end{proof}

\section{Proof of Theorems \ref{pricipal-1} and \ref{pricipal-2}}\label{cap-existence}

Consider the associated functional to Problem \eqref{Pa} $I:W_0^s L^\Phi (\Omega)\to \mathbb{R}\cup\{+\infty\}$, defined by
$$I(u)=(1-s)\iint\limits_{\mathbb{R}^N \times \mathbb{R}^N}\Phi(|D_su|)\mathrm{d}\mu-\frac{1}{1-\gamma}\int\limits_{\Omega} |u|^{1-\gamma}\mathrm{d}x:=I_1(u)+I_2(u),~u\in W_0^s L^\Phi(\Omega),$$
where $I_1: W_0^s L^\Phi(\Omega)\to \mathbb{R}\cup \{+\infty\}$ and $I_2: L^\Phi(\Omega) \to \mathbb{R}$ are defined by
$$I_1(u):=(1-s)\iint\limits_{\mathbb{R}^N \times \mathbb{R}^N}\Phi(|D_su|)\mathrm{d}\mu~\mbox{and}~I_2(u):=-\frac{1}{1-\gamma}\int\limits_{\Omega} |u|^{1-\gamma}\mathrm{d}x$$
with $0<s,\gamma<1$.

\begin{prop}\label{I2continuous}
$I_2$ is well-defined on $L^\Phi(\Omega)$ and a continuous functional.
\end{prop}
\begin{proof}
First, we note that $I_2$ is well defined. In fact, it follows from  Holder's inequality  and Proposition \ref{LPhicontinuousL1} that  
\begin{eqnarray}\label{L1}
\begin{array}{lcl}
    \displaystyle \int\limits_{\Omega}|u|^{1-\gamma}\mathrm{d}x &  \leq  & \displaystyle  \left[\int\limits_\Omega 1^{\frac{1}{\gamma}}\mathrm{d}x \right]^\gamma \left[\int\limits_\Omega |u|^{1-\gamma \frac{1}{1-\gamma}}\mathrm{d}x \right]^{1-\gamma} \\
     &  =  & \displaystyle  |\Omega|^\gamma ||u||_{L^1(\Omega)}^{1-\gamma} 
       \leq   \displaystyle \tilde{C}  ||u||_\Phi^{1-\gamma} \\
\end{array}
\end{eqnarray}  
for all $u \in L^\Phi(\Omega)$, where $0<\tilde{C}=|\Omega|^\gamma k< + \infty$ with $k>0$ being Poincaré's constant.

Furthermore, we have that $I_2$ is continuous. In fact, let $u_n \rightarrow u$ in $L^\Phi(\Omega)$. Then 
$$\displaystyle \int\limits_\Omega \Phi \left(\frac{u_n -u}{\lambda} \right)\mathrm{d}x \longrightarrow 0, \forall \lambda>0,$$
that is,
 $\displaystyle \Phi \left({u_n -u} \right) \longrightarrow 0$ in $L^1(\Omega)$. So, by 
inverse of the Dominated Convergence Theorem
there exists $h \in L^1(\Omega)$ such that 
\begin{equation}\label{3.1}
|u_n - u| \leq \Phi^{-1}(h).    
\end{equation}

Now, since $\displaystyle \lim_{t \rightarrow +\infty}{\Phi(t)}/t=+\infty$, there exists $T>0$ such that $t \leq \Phi(t), \forall t>T$, or equivalently 
\begin{equation}\label{3.2}
\Phi^{-1}(t) \leq t, \forall t >T, 
\end{equation}
which implies together with \eqref{3.1}, \eqref{3.2} and   $t \mapsto t^{1-\gamma}$ be concave for $t>0$, that
$$|u_n|^{1-\gamma}\leq |u_n-u|^{1-\gamma}+|u|^{1-\gamma} \leq [\Phi^{-1}(h)]^{1-\gamma}+|u|^{1-\gamma} \leq h^{1-\gamma}+|u|^{1-\gamma},$$
whence follows that $|u|^{1-\gamma} \in L^1(\Omega)$ due to 
 $h \in L^1(\Omega)$
and \eqref{L1}. So, we are ready to use the Dominated Convergence Theorem to obtain 
$$\displaystyle \lim_{n \rightarrow +\infty}\int\limits_\Omega |u_n|^{1-\gamma}\mathrm{d}x =\int\limits_\Omega  \lim_{n \rightarrow +\infty}|u_n|^{1-\gamma}\mathrm{d}x = \int\limits_\Omega |u|^{1-\gamma}\mathrm{d}x,$$
that is, $I_2(u_n) \rightarrow I_2(u)$ in $\mathbb{R}$. This ends the proof.
\end{proof}

Below, we point out some consequences on the energy functional from the fact that $\Phi$ may not belongs to $\Delta_2$. First, we note that under our hypotheses, it is possible 
$$m=\displaystyle \sup_{t>0}\frac{\phi(t)t^2}{\Phi(t)}=\infty$$ 
to occur so that $I_1(u)<\infty$ for all $u \in W_0^s L^\Phi(\Omega)$ does not occur anymore. From this, we need to define the effective domain 
$$D_\Phi:=\{u\in W_0^s L^\Phi(\Omega):~I_1(u)<\infty\}$$
to work variationally on it.  Just to remind,  we have $D_\Phi=  W_0^s L^\Phi (\Omega)$ if, and only if, $\Phi\in \Delta_2$.

Second, our assumptions admits the situation 
$${\tilde{\ell}}=\inf\limits_{t>0}\frac{\tilde{\Phi}'(t)t}{\tilde{\Phi}(t)} = 1$$
as well, whence is not possible to guarantee the finitude of the second integral below 
$$\iint\limits_{\mathbb{R}^N \times \mathbb{R}^N}\tilde{\Phi}(\phi(|D_s u||D_s u|)\mathrm{d}\mu \leq \iint\limits_{\mathbb{R}^N \times \mathbb{R}^N}\Phi(2|D_s u|)\mathrm{d}\mu $$
anymore. From now on, let us denote by 
 $D_{\Tilde{\Phi}} \subset  W^s_0 L^\Phi(\Omega)$ the set
$$
\begin{array}{ccl}
     D_{\Tilde{\Phi}} &=& \displaystyle  \left\{ u \in  W^s_0 L^\Phi(\Omega)\,:\, (1-s) \iint\limits_{\mathbb{R}^N \times \mathbb{R}^N}\widetilde{\Phi}(\phi(|D_s u|)|D_su |)\,d\mu<\infty \right\}\\\\
     & = & \displaystyle \left\{ u \in   W^s_0 L^\Phi(\Omega)\,:\, \phi(|D_s u|)|D_s u|  \in \mathcal{L}^{\tilde{\Phi}}(\Omega) \right\}.
\end{array}
$$
It is interesting to point out that if $\Phi \in \Delta_2(\infty)$ then $D_{\Tilde{\Phi}}=W_0^sL^\Phi(\Omega)$.

\begin{rmk} One has:
\begin{enumerate}
    \item[i)] $C_0^{\infty}(\Omega) \subset D_{\Tilde{\Phi}}$,
    \item[ii)] $ D_\Phi$ is a convex subset of $W^s_0 L^\Phi(\Omega)$.
\end{enumerate}
\end{rmk}

The below lemmas will ease the proof of the Theorem \ref{pricipal-1}.

\begin{lem}\label{coercphi3}Assume the assumptions of Theorem \ref{pricipal-1}.
Then $I$ is coercive and so bounded from below.
\end{lem}

\begin{proof}
Since $\displaystyle\lim_{t\to \infty}\frac{\Phi(\lambda t)}{t}=\infty$, for all $\lambda>0$,  there exists $T_\lambda>0$ such that     
$$\Phi(\lambda t)\geq  t,~t\geq T_\lambda.$$
From this, Holder's Inequality and Proposition \ref{Poincare} we deduce
\begin{eqnarray}
    \int\limits_\Omega |u|^{1-\gamma}dx &\leq& \int\limits_{[|u|\geq T_\lambda]} \Phi(\lambda |u|)^{1-\gamma}dx+|\Omega|T_\lambda^{1-\gamma}\nonumber\\
    &\leq& {|\Omega|^\gamma}\left(\int\limits_{\Omega} \Phi(\lambda |u|)dx\right)^{1-\gamma}+|\Omega|T_\lambda^{1-\gamma} \nonumber\\
    &\leq& {|\Omega|^\gamma}\left( \ \ 
 \iint\limits_{\mathbb{R}^N \times \mathbb{R}^N}\Phi(\lambda C|D_s u|)\mathrm{d}\mu\right)^{1-\gamma}+|\Omega|T_\lambda^{1-\gamma} \nonumber.
\end{eqnarray}
Making $\lambda=\frac{1}{C}$, we conclude that
\begin{eqnarray}
    \int\limits_\Omega |u|^{1-\gamma}dx\leq \frac{|\Omega|^\gamma}{(1-s)^{1-\gamma}}I_1(u)^{1-\gamma}+|\Omega|T_\lambda^{1-\gamma}\nonumber.
\end{eqnarray}
This implies that 
\begin{eqnarray}\label{des-coer}
    \displaystyle I(u) &\geq& I_1(u)
-\frac{|\Omega|^\gamma}{(1-\gamma)(1-s)^{1-\gamma}}I_1(u)^{1-\gamma}-\frac{|\Omega|T_\lambda^{1-\gamma}}{1-\gamma},
\end{eqnarray}
The inequality \eqref{des-coer} and Proposition \ref{prop2.3} lead to the coercivity of $I$, because $ 1 > 1-\gamma$.
\end{proof}

%
%

After this, we have well-defined the infimum of $I$ over $ W^s_0  L^\Phi(\Omega)$.

\begin{lem}\label{negative-sing}
Assume the assumptions of Theorem \ref{pricipal-1}. Then:
\begin{enumerate}
    \item[i)] the $\inf_{ W^s_0  L^\Phi(\Omega)}I$ is negative,
    \item[ii)] there exists a non-null global minimizer $0\leq u \in W^s_0 L^\Phi(\Omega)$ of $I$. In particular, $u \in D_\Phi$.
\end{enumerate}
\end{lem}

\begin{proof} 
Let {$0 \neq u \in  W^s_0 L^\Phi (\Omega)$}. By using the fact that $\Phi$ is convex and $\gamma, t \in (0,1)$, we obtain
\begin{eqnarray*}
\begin{array}{lcl}
    \displaystyle \lim_{t \rightarrow 0}I(tu)     & = & \displaystyle \lim_{t \rightarrow 0} \left[ \ (1-s)\iint\limits_{\mathbb{R}^N \times \mathbb{R}^N}\Phi\left(\frac{tu(x)-tu(y)}{|x-y|^s} \right) \mathrm{d}\mu - \frac{1}{1-\gamma} \int\limits_\Omega|tu|^{1-\gamma} \mathrm{d}x \right] \\
    
     & \leq  & \displaystyle \lim_{t \rightarrow 0} \left[t(1-s) \iint\limits_{\mathbb{R}^N \times \mathbb{R}^N}\Phi\left(\frac{u(x)-u(y)}{|x-y|^s} \right) \mathrm{d}\mu - \frac{t^{1-\gamma}}{1-\gamma} \int\limits_\Omega |u|^{1-\gamma}\mathrm{d}x \right] <0,
\end{array}
\end{eqnarray*}
so that $I(tu)<0$ for $t>0$ small enough. This proves the item i).

Due to Lemmas \ref{coercphi3} and the item $i)$ just proved, there exists a minimizer sequence  $(u_n) \subset W^s_0 L^\Phi(\Omega)$ to the functional $I$, which implies by the coercivity of $I$, that 
$(u_n)$ is bounded in $W^s_0 L^\Phi(\Omega)$. Therefore, by Lemma \ref{Estrela}, for some subsequence denoted by itself, we obtain
$$
u_n \stackrel{*}{\rightharpoonup} u \quad \mbox{in} \quad  W^s_0 L^\Phi(\Omega),$$ 
that leads to 
$$
\inf_{ W^s_0  L^\Phi(\Omega)}I=\liminf_{n \to +\infty}I(u_n) \geq I(u)
$$
due to the fact that $I_1$ is weak$^*$ lower semicontinuous, as proved in \cite{Salort}. That is, $u$ is a non-null global minimizer that can be considered non-negative due to the parity of the energy functional $I$. 
Proposition \ref{211} together with \eqref{L1} imply $$I_1(u)<I_2(u) \leq \tilde{C}\Vert u \Vert_{\Phi}^{1-\gamma} \leq C[u ]_{s,\Phi}^{1-\gamma}$$
showing that $u \in D_{\Phi}$. This ends the proof.
\end{proof}

Below, we inspired on ideas from \cite[Lemma 3.4]{AlvesCarvalho} to prove, in particular, that the functional $I$ is finite at $u \in D_{\Phi}$ given above.

\begin{lem}\label{Domin}
Assume the assumptions of Theorem \ref{pricipal-1}. Let
$u\in D_\Phi$ be that given by Lemma  \ref{negative-sing}. Then
\begin{eqnarray}\label{des1-u1}
    (1-s)\iint\limits_{\mathbb{R}^N \times \mathbb{R}^N}\phi(|D_s u|)|D_s u|^2\mathrm{d}\mu \leq \int\limits_\Omega u^{1-\gamma}\mathrm{d}x.
\end{eqnarray}
  In particular, $u \in D_{\Tilde{\Phi}}$, that is, $\phi(|D_s u|)|D_s u| \in L^{\tilde{\Phi}}(\Omega).$
\end{lem}

\begin{proof} Since $u$ is a global minimizer of $I$, that is,
 $$\displaystyle  n[{I_1(u)-I_1(v)}] \leq n[{I_2(v)-I_2(u)}]~\mbox{for all }v  \in W^s_0 L^\Phi(\Omega),~n\in \mathbb{N},$$
whence follows, together with $D_\Phi$ be convex, that
\begin{equation}\label{desi1}
\displaystyle n(1-s) \left( \ \iint\limits_{\mathbb{R}^N \times \mathbb{R}^N} \Phi(|D_s u|) \mathrm{d}\mu - \iint\limits_{\mathbb{R}^N \times \mathbb{R}^N}\Phi\left(\left| \left(1-\frac{1}{n} \right) D_s u \right| \right) \mathrm{d}\mu \right) \leq \frac{-1}{1-\gamma} \left[\frac{\left( 1-\frac{1}{n}\right)^{1-\gamma}-1}{\frac{1}{n}}\right]\int\limits_\Omega |u|^{1-\gamma}\mathrm{d}x
\end{equation}
by taking then $\displaystyle v = \left(1-{1}/{n} \right)u \in D_\Phi$ for any $n \in \mathbb{N}$. 

By using that  $\Phi \in C^1$, there exists a $\beta_n \in [0,1]$ such that
\begin{eqnarray*}
\begin{array}{lcl}
\displaystyle n (1-s)\left[\Phi(|D_s u|)-\Phi \left(\left| \left(1-\frac{1}{n} \right)D_s u\right| \right)\right] &  =  & 

\displaystyle (1-s)\Phi'\left(\left|\left(1-\frac{\beta_n}{n}\right)D_s u\right| \right)|D_s u| \\

&  =  & \displaystyle (1-s)\phi \left(\left|\left(1-\frac{\beta_n}{n}\right)D_s u\right| \right)\left|1-\frac{\beta_n}{n} \right||D_s u|^2,
\end{array}
\end{eqnarray*}
whence follows that 
\begin{equation}\label{desi2}
\displaystyle (1-s)\phi \left(\left(1-\frac{\beta_n}{n} \right)|D_s u| \right)\left|1-\frac{\beta_n}{n} \right||D_s u|^2 \geq (1-s)\phi \left(\left(1-\frac{\beta_n}{n} \right)|D_s u| \right)\left|\left(1-\frac{\beta_n}{n}\right)D_s u \ \right|^2
\end{equation}
due to the fact that  $\displaystyle 1-{\beta_n}/{n} \geq \left( 1-\beta_n/{n} \right)^2$. 

So, it is a consequence of \eqref{desi1} and \eqref{desi2} that
$$(1-s)\iint\limits_{\mathbb{R}^N \times \mathbb{R}^N} \phi\left(\left(1-\frac{\beta_n}{n} \right)|D_s u| \right)\left|\left(1-\frac{\beta_n}{n} \right)D_s u \right|^2 \mathrm{d}\mu \leq \frac{-1}{1-\gamma} \left[\frac{\left( 1-\frac{1}{n}\right)^{1-\gamma}-1}{\frac{1}{n}}\right]\int\limits_\Omega |u|^{1-\gamma}\mathrm{d}x,$$
which implies 
\begin{eqnarray}\label{dom}
    (1-s)\iint\limits_{\mathbb{R}^N \times \mathbb{R}^N}\phi(|D_s u|)|D_s u|^2\mathrm{d}\mu \leq 
    \int\limits_\Omega |u|^{1-\gamma}\mathrm{d}x,
\end{eqnarray}
after using Fatou's Lemma.

To complete the proof, just remains to prove that $u \in D_{\Tilde{\Phi}}$. To do this, we obtain from Young's equality that
$$\phi(|D_s u|)D_s u D_s u =\Phi(|D_s u|) + \tilde{\Phi}(\phi(|D_s u|)|D_s u|)$$
so that
$$\displaystyle (1-s)\iint\limits_{\mathbb{R}^N \times \mathbb{R}^N}\phi(|D_s u|)|D_s u|^2 \mathrm{d}\mu = I_1(u) + (1-s)\iint\limits_{\mathbb{R}^N \times\mathbb{R}^N}\tilde{\Phi}(\phi(|D_s u|)|D_s u|)\mathrm{d}\mu,$$
that implies
$$\displaystyle (1-s)\iint\limits_{\mathbb{R}^N \times \mathbb{R}^N} \tilde{\Phi}(\phi(|D_s u|)|D_s u|)\mathrm{d}\mu < +\infty,$$
after using $u \in D_\Phi$, \eqref{dom} and \eqref{L1}. This shows that $u \in D_{\Tilde{\Phi}}$, and ends the proof.
\end{proof}


The next result guarantees that the global minimizer $u\in D_\Phi \cap D_{\Tilde{\Phi}}$ of $I$ is almost everywhere (a.e.) positive and satisfies the inequality \eqref{14c} below that is very important to show that $u$ is a weak solution to Problem \eqref{Pa}. 
 
 To prove the below Lemma, we carefully build a test function $v_t$ that goes to $u$ as $t\to 0$ inside the effective domain of the non-differentiable functional of $I$. This kind of idea has been used in different contexts for functionals and workspaces with less restrictions, here we do note have the guarantee that $\Phi \in \Delta_2$. See for instance \cite{EGS,Mustonen,Yijing}.
 
\begin{lem}\label{solpos}
Assume the assumptions of Theorem \ref{pricipal-1}. Then $u(x)>0$ a.e. in $\Omega$,  $u^{-\gamma} v\in L^1(\Omega),~\forall v\in W_0^sL^\Phi(\Omega)$ and
\begin{eqnarray}\label{inequality-fund}
\begin{array}{lcl}\label{14c}
    \displaystyle (1-s)\iint\limits_{\mathbb{R}^N \times \mathbb{R}^N}\phi(|D_s u|)D_s u D_s\left(-2u+(v-u)^+ \right) \mathrm{d}\mu & \geq & \displaystyle \int_\Omega u^{-\gamma}\left(-2u + (v-u)^+ \right)\mathrm{d}x, 
\end{array}    
\end{eqnarray}
for all $0\leq v \in W_0^sL^\Phi(\Omega)$.
\end{lem}
\begin{proof}
Let $v \in C_0^\infty(\Omega), v\geq 0 $. For each ${0<t<{1}/{3}}$, the function $$\displaystyle v_t = (1-2t)u+t(v-u)^+$$
satisfies
$$0<v_t\leq (1-2t)u+t|v-u|\leq (1-2t)u+tv+tu=(1-t)u+tv\in \mathcal{L}_\Phi (\Omega),$$
due to the convexity of $\Phi$. Besides this, it follows from Lemma \ref{u+} that $|D_s(u-v)^+|\leq  |D_su|+|D_sv|.$ Then, using this fact and the monotonicity of $\phi(|\cdot|)|\cdot|$ and  we have that
$$\left|D_sv_t\right|\leq\left|D_s\left((1-2t)u\right)\right|+\left|D_s\left(t(v-u)^+\right)\right| \leq (1-2t)\left|D_su\right|+t\left|D_sv\right|+t\left|D_su\right| \leq (1-t)\left|D_su\right|+t\left|D_sv\right|,$$  
whence follows
$$I_1(v_t) \leq (1-t)\iint\limits_{\mathbb{R}^N \times \mathbb{R}^N}\Phi(|D_s u|)\mathrm{d}\mu + t\iint\limits_{\mathbb{R}^N \times \mathbb{R}^N}\Phi(|D_s v|)\mathrm{d}\mu < +\infty,$$
due to the convexity of $\Phi$, that is, $v_t \in D_\Phi.$

As a consequence of $v_t \in D_\Phi$ and $u \in D_\Phi$ be a global minimizer of $I$, we have
\begin{eqnarray}\label{0.1}
\begin{array}{lcl}
    \displaystyle \frac{1}{t}  \int\limits_\Omega \frac{v_{t}^{1-\gamma}-u^{1-\gamma}}{1-\gamma}\mathrm{d}x \leq \displaystyle (1-s)  \iint\limits_{\mathbb{R}^N \times \mathbb{R}^N}\frac{\Phi(|D_s v_t|)-\Phi(|D_s u|)}{t}\mathrm{d}\mu
\end{array}
\end{eqnarray}
for any $t>0$. 

Below, we will show that there  exist and calculate the above limits via  Dominated Convergence Theorem.  We begin proving that
\begin{equation}
    \label{17a}
    \displaystyle \left|\frac{\Phi(|D_s v_t|)-\Phi(|D_s u|)}{t} \right| \leq h_1,
\end{equation}
for some $h_1 \in L^1(\mathbb{R}^N \times \mathbb{R}^N, \mathrm{d}\mu)$. In fact, it follows from Lemma \ref{Phi}-(i) that
\begin{eqnarray}\label{0.3}
\begin{array}{ll}
  \displaystyle \left|\frac{\Phi(|D_s v_t|)-\Phi(|D_s u|)}{t} \right|   & \leq [\phi(|D_s v_t|)|D_s v_t|+\phi(|D_s u|)|D_s u|]\frac{|D_s v_t - D_s u|}{t} \\
  \\
     & := A + B  \phi(|D_s u|)|D_s u|,
\end{array}
\end{eqnarray}
where 
$$A=\phi(|D_s v_t|)|D_s v_t|~\mbox{and}~ B= \frac{|D_s v_t - D_s u|}{t}.$$

Now, let us estimate $A$ and $B$. It follows from the monotonicity of $\phi(|\cdot|)|\cdot|$ and the fact $|D_s(u-v)^+|\leq  |D_su|+|D_sv|,$ that
$$\left|D_sv_t\right|\leq\left|D_s\left((1-2t)u\right)\right|+\left|D_s\left(t(v-u)^+\right)\right| \leq (1-2t)\left|D_su\right|+t\left|D_sv\right|+t\left|D_su\right| \leq (1-t)\left|D_su\right|+t\left|D_sv\right|,$$  
whence follows by $\varphi(|\cdot|)|\cdot|$ be increasing and Lemma \ref{Phi}-(ii), that
\begin{eqnarray}\label{0.4}
\begin{array}{lcl}
\displaystyle A &\leq &\ \displaystyle \phi\left(\left|D_su\right|\right)\left|D_su\right|+ \phi\left(\left|D_sv\right|\right)\left|D_sv\right|.
\end{array}
\end{eqnarray}

About $B$, it is immediate that
\begin{eqnarray}\label{0.5}
\begin{array}{lcccl}
\displaystyle B:=\frac{\left|D_sv_{t}-D_su_{1}\right|}{t}&=&\displaystyle \frac{\left|-2tD_su_{1}+tD_s(v-u)^+\right|}{t}&\leq& 3\left|D_su_{1}\right|+\left|D_sv\right|.
\end{array}
\end{eqnarray}

So, by using \eqref{0.4} and \eqref{0.5} in \eqref{0.3}, we obtain 
\begin{eqnarray*}\label{h_1}
\begin{array}{lcl}
    \displaystyle \left|\frac{\Phi(|D_s v_t|)-\Phi(|D_s u|)}{t} \right| & \leq & \displaystyle ( 2\phi(|D_su|)|D_su|+ \phi(|D_sv|)|D_sv|)(3|D_su_{1}|+|D_sv|) \\
    
     & =  & \displaystyle 6\phi(|D_s u|)|D_s u|^2 + 
     \phi(|D_s v|)(|D_s v|)^2 \\ 
     
     &+& \displaystyle 2\phi(|D_s u|)|D_s u||D_s v| + 3\phi(|D_s v|)|D_s v||D_s u| =:h_1  .
\end{array}
\end{eqnarray*}

As $u \in W^sL^\Phi(\Omega)$ and $v \in C_0^\infty (\Omega),$ the first two terms of $h_1$ belong to $L^1(\mathbb{R}^N \times \mathbb{R}^N, \mathrm{d}\mu)$, while the conclusion of the last two terms of $h_1$ to belong to $L^1(\mathbb{R}^N \times \mathbb{R}^N, \mathrm{d}\mu)$ is a consequence of Young's inequality and that $u \in D_{\tilde{\Phi}}$, that is, $h_1 \in L^1(\mathbb{R}^N \times \mathbb{R}^N,\mathrm{d}\mu)$, as claimed. 


After these, we able to apply Dominated Convergence Theorem in the last term of \eqref{0.1} to obtain
\begin{eqnarray}\label{0.6}
\begin{array}{lcl}
    \displaystyle (1-s)\lim_{t \rightarrow 0} \iint\limits_{\mathbb{R}^N \times \mathbb{R}^N}\frac{\Phi(|D_s v_t|)-\Phi(|D_s u|)}{t} \mathrm{d}\mu & = &  \displaystyle (1-s)\iint\limits_{\mathbb{R}^N \times \mathbb{R}^N}\lim_{t \rightarrow 0} \frac{\Phi(|D_s v_t|)-\Phi(|D_s u|)}{t} \mathrm{d}\mu \\
     & = & \displaystyle (1-s)\iint\limits_{\mathbb{R}^N \times \mathbb{R}^N} \lim_{t \rightarrow 0}\frac{g(t)-g(0)}{t} \mathrm{d}\mu,
\end{array}
\end{eqnarray}
where $g(t)=\Phi(|D_s v_t|), t \geq 0,$ for each $x,y \in \mathbb{R}^N$.  Since,  
$$ \lim_{t \rightarrow 0}\frac{g(t)-g(0)}{t} =\displaystyle \phi(|D_s u|)D_s u D_s \left(-2u+(v-u)^+ \right)~\mbox{for each}~x,y \in \mathbb{R}^N,$$
we obtain
\begin{eqnarray}\label{finite}
\begin{array}{lcl}
     \displaystyle (1-s)\lim_{t \rightarrow 0} \iint\limits_{\mathbb{R}^N \times \mathbb{R}^N}\frac{\Phi(|D_s v_t|)-\Phi(|D_s u|)}{t} \mathrm{d}\mu
    & =  & \displaystyle (1-s)\iint\limits_{\mathbb{R}^N \times \mathbb{R}^N}\phi(|D_s u|)D_s u D_s\left(-2u + (v-u)^+ \right) \mathrm{d}\mu.
\end{array}
\end{eqnarray}


Besides this, it follows from  $|u^+(x)-u^+(y)| \leq |u(x)-u(y)|$, $u \in D_{\Tilde{\Phi}}$, and Young's inequality, that
\begin{eqnarray}\label{3.17}
\begin{array}{lcl}
    \displaystyle (1-s)\iint\limits_{\mathbb{R}^N \times \mathbb{R}^N}\phi(|D_s u|)D_s u D_s\left(-2u + (v-u)^+ \right) \mathrm{d}\mu &  \leq  & \displaystyle(1-s) \left(3 \iint\limits_{\mathbb{R}^N \times \mathbb{R}^N}\phi(|D_s u|)|D_s u|^2 \mathrm{d}\mu \right.\\
    
    & + & \displaystyle \left.\iint\limits_{\mathbb{R}^N \times \mathbb{R}^N}\phi(|D_s u|)|D_s u||D_s v| \mathrm{d}\mu\right) <+\infty,
\end{array}
\end{eqnarray}
for all $v \in C_0^\infty(\Omega)$ with $ v \geq 0$, that is,
$$\displaystyle \iint\limits_{\mathbb{R}^N \times \mathbb{R}^N}\phi(|D_s u|)|D_s u|D_s v \mathrm{d}\mu \leq \iint\limits_{\mathbb{R}^N \times \mathbb{R}^N} \tilde{\Phi}(\phi(|D_s u|)|D_s u|)\mathrm{d}\mu + \iint\limits_{\mathbb{R}^N \times \mathbb{R}^N} \Phi(|D_s v|)\mathrm{d}\mu < +\infty.$$ 

Below, let us estimate the left side of \eqref{0.1}. We will do this, by splitting our integration domain as follows 
\begin{equation}
    \label{24a}
    \displaystyle \frac{1}{t} \int\limits_\Omega \frac{v_{t}^{1-\gamma}-u^{1-\gamma}}{1-\gamma}\mathrm{d}x = \frac{1}{t} \left[ \int\limits_{[v \leq u]} + \int\limits_{[u < v< 2u]} + \int\limits_{[v > 2u]} \right]\frac{v_{t}^{1-\gamma}-u^{1-\gamma}}{1-\gamma}\mathrm{d}x, 
\end{equation}
and evaluating each integral.

First for $[v \leq u]$, we have straightforward from the definition of $v_t$, that 
\begin{equation}\label{vmenoru1}
\displaystyle \lim_{t \rightarrow 0}\frac{1}{1-\gamma}\int\limits_{[v \leq u]}\frac{v_t^{1-\gamma}-u^{1-\gamma}}{t}\mathrm{d}t    =\lim_{t \rightarrow 0}\frac{1}{1-\gamma}\frac{(1-2t)^{1-\gamma}-1}{t}\int\limits_{[v \leq u]}u^{1-\gamma}\mathrm{d}x = -2  \int\limits_{[v \leq u]} u^{1-\gamma}\mathrm{d}x,
\end{equation}
while for $[u < v < 2u]$, we note that there exists  a $\bar{t}=\bar{t}(x) \in (0,1)$ such that 
$$\displaystyle \left|\frac{1}{1-\gamma} \frac{g_2(t)-g_2(0)}{t}\right| = \frac{|g'_2 (\bar{t})|}{1-\gamma}\leq Ku^{1-\gamma} \in L^1(\Omega),$$
for some $K>0$, where $g_2(t):=[(1-3t)u+tv]^{1-\gamma}$  for $t \geq 0$. So, it follows by the Dominated Convergence Theorem that
\begin{eqnarray}\label{u1menorvmenor2u1}
\begin{array}{lcl}
    \displaystyle \lim_{t \rightarrow 0}\frac{1}{1-\gamma}\int\limits_{[u < v < 2u]}\frac{v_t^{1-\gamma}-u^{1-\gamma}}{t}\mathrm{d}t   &  = & \displaystyle \int\limits_{[u < v < 2u]}\lim_{t \rightarrow 0}((1-3t)u + tv)^{-\gamma}(-3u + v) \mathrm{d}x \\
    
    & = & \displaystyle \int\limits_{[u < v < 2u]} (-3u^{1-\gamma}+ u^{-\gamma}v) \mathrm{d}x.
\end{array}
\end{eqnarray}

Finally, let us consider $[v > 2u]$. By noting that  
$$\displaystyle \frac{1}{1-\gamma} \frac{g_2(t)-g_2(0)}{t} = \frac{g'_2 (\bar{t})}{1-\gamma}= ((1-3t)u + tv)^{-\gamma}(-3u + v)>-u^{1-\gamma} \in L^1(\Omega),$$ 
one has  by Fatou's Lemma 
\begin{eqnarray}\label{vmaioru1}
\begin{array}{lcl}
    \displaystyle \liminf_{t \rightarrow 0}\frac{1}{1-\gamma} \int\limits_{[v > 2u]}\frac{v_t^{1-\gamma}-u^{1-\gamma}}{t}\mathrm{d}x &  \geq & \displaystyle \int\limits_{[v >2u]}\liminf_{t \rightarrow 0} ((1-3t)u +tv)^{-\gamma}(-3u+v)\mathrm{d}x \\
    
    & = & \displaystyle  \int\limits_{[v > 2u]} (-3u^{1-\gamma}+  u^{-\gamma}v) \mathrm{d}x.
\end{array}
\end{eqnarray}

So, by using \eqref{vmenoru1}-\eqref{vmaioru1} in \eqref{24a}, we get to 
\begin{eqnarray}\label{0.9}
\begin{array}{lcl}
    \displaystyle \limsup_{t \rightarrow 0}\frac{1}{t}\int\limits_\Omega \frac{v_t^{1-\gamma} - u^{1-\gamma}}{1-\gamma}\mathrm{d}x &\geq &  \displaystyle  \int\limits_{[v \leq u]} -2u^{1-\gamma} \mathrm{d}x + \left[ \int\limits_{[u < v< 2u]} +\int\limits_{[v > 2u]} \right] (-3u^{1-\gamma}+u^{-\gamma}v) \mathrm{d}x\\
    &=&\displaystyle \int\limits_\Omega u^{-\gamma}(-2u+(v-u)^+)\mathrm{d}x ,
    
\end{array}     
\end{eqnarray}
where the equality is just a re-written of the former terms.

So, it is a consequence of \eqref{0.1}, \eqref{finite},  \eqref{3.17} and \eqref{0.9} that
\begin{eqnarray}\label{vgeq0}
\begin{array}{lcl}
    \displaystyle+\infty > (1-s)\iint\limits_{\mathbb{R}^N \times \mathbb{R}^N}\phi(|D_s u|)D_s u D_s\left(-2u+(v-u)^+ \right) \mathrm{d}\mu & \geq & \displaystyle \int_\Omega u^{-\gamma}\left(-2u + (v-u)^+ \right)\mathrm{d}x
\end{array}    
\end{eqnarray}
for all $v\in C^\infty_0(\Omega)$ with $v\geq 0$.

Below, let us prove that $u(x) >0$ a.e. in $\Omega$. First, we note that follows from \eqref{L1} that
\begin{equation}\label{0.10}
\displaystyle 0 \leq \int\limits_{[v \leq u]}u^{-\gamma}v\mathrm{d}x \leq \int\limits_{[v \leq u]}u^{1-\gamma}\mathrm{d}x < +\infty.
\end{equation}

In addition, it is a consequence of \eqref{3.17}, \eqref{0.9} and \eqref{vgeq0},  that

\begin{equation}\label{0.12}
\displaystyle +\infty >  (1-s)\iint\limits_{\mathbb{R}^N \times \mathbb{R}^N}\phi(|D_s u|)D_s u D_s\left(-2u+(v-u)^+ \right) \mathrm{d}\mu+2 \int\limits_{[v \leq u]} u^{1-\gamma} \mathrm{d}x + 3\int\limits_{[v>u]}u^{1-\gamma}\mathrm{d}x \geq \int\limits_{[v>u]}u^{-\gamma}v \mathrm{d}x.    
\end{equation}

Therefore, by \eqref{0.10} and \eqref{0.12} we have 
\begin{equation}\label{singularfinite}
\int\limits_\Omega u^{-\gamma} v\mathrm{d}x = \left[ \  \int\limits_{[v \leq u]} + \int\limits_{[v>u]}  \right]u^{-\gamma}v\mathrm{d}x < +\infty, \ \forall v \in C_0^\infty (\Omega), v \geq 0.   
\end{equation}

That is, there exists $M_v >0$ such that 
\begin{equation}\label{0.11}
\displaystyle 0 \leq \int\limits_\Omega \frac{v}{u^\gamma}\mathrm{d}x < M_v.    
\end{equation}

Now, assume by contradiction that the set $\Omega_0 = \{x \in \Omega; u (x) \equiv 0 \}$ has no null Lebesgue's measure. So,
$$\displaystyle \int\limits_\Omega \frac{v}{u^\gamma}\mathrm{d}x = \left[\int\limits_{[u < \varepsilon ]} + \int\limits_{[u \geq \varepsilon ]} \right] \frac{v}{u^\gamma}\mathrm{d}x > \int\limits_{[u < \varepsilon ]}\frac{v}{u^\gamma}\mathrm{d}x > \frac{1}{\varepsilon^\gamma}\int\limits_{\Omega_0} v \mathrm{d}x \longrightarrow +\infty,$$ as $\varepsilon \rightarrow 0,$
what contradicts \eqref{0.11}. Hence, we have $u(x) >0$ a.e. in $\Omega.$
Moreover, it follows from \eqref{0.11} that
$$\displaystyle \left|\int\limits_\Omega u^{-\gamma}v \mathrm{d}x \right|= \left|\int\limits_\Omega u^{-\gamma}(v^+ - v^-) \mathrm{d}x \right| \leq \int\limits_\Omega u^{-\gamma}v^+ \mathrm{d}x + \int\limits_\Omega u^{-\gamma}v^- \mathrm{d}x < +\infty,$$
that is, $u^{-\gamma}v \in L^1(\Omega), \ \ \forall v \in C_0^\infty(\Omega)$. 

Finally, let us show \eqref{14c}. Given $v \in W_0^sL^\Phi(\Omega)$ such that $v \geq 0$, we know from Proposition \ref{impontant} that there exists $(v_n) \in C_0^\infty (\Omega), v_n \geq 0$ such that $v_n \stackrel{*}{\rightharpoonup} v.$ From this convergence and \eqref{vgeq0}, we obtain
\begin{eqnarray}\label{desig0}
    \displaystyle+\infty > \iint\limits_{\mathbb{R}^N \times \mathbb{R}^N}\phi(|D_s u|)D_s u D_s\left(-2u+(v-u)^+ \right) \mathrm{d}\mu & = & \lim_{n\to+\infty}\iint\limits_{\mathbb{R}^N \times \mathbb{R}^N}\phi(|D_s u|)D_s u D_s\left(-2u+(v_n-u)^+ \right) \mathrm{d}\mu\nonumber\\
    &\geq & \frac{1}{1-s}\liminf_{n\to+\infty}\int_\Omega u^{-\gamma}\left(-2u + (v_n-u)^+ \right)\mathrm{d}x\nonumber\\
    &\geq & \displaystyle \int\limits_\Omega u^{-\gamma}(-2u+(v-u)^+)\mathrm{d}x,
\end{eqnarray}
where the last inequality follows from Fatou's Lemma. That is, the equation \eqref{inequality-fund} holds for all $ v \in W_0^sL^\Phi(\Omega).$ This ends the proof.
\end{proof}

As a corollary of the above lemma we have the following result. Such corollary will be an important tool to prove the Proposition \ref{propoprincipal}.

\begin{cor}\label{igual}
Assume the assumptions of Theorem \ref{pricipal-1}. Then $u \in W_0^s L^\Phi (\Omega)$ satisfies the following
$$(1-s) \iint\limits_{\mathbb{R}^N \times \mathbb{R}^N}\phi(|D_s u|)|D_s u|^2\, \mathrm{d}\mu =\int\limits_{\Omega}u^{1-\gamma}\,dx.$$
\end{cor}




\begin{proof}
First, we note that doing  $v=4u$ in \eqref{inequality-fund}, we obtain
$$(1-s)\iint\limits_{\mathbb{R}^N \times \mathbb{R}^N}\phi(|D_s u|)|D_s u|^2  \mathrm{d}\mu\geq \int\limits_{\Omega}u^{1-\gamma} \mathrm{d}x,$$
which implies the equality after using the \eqref{des1-u1}, that is, the corollary is proved.
\end{proof}

The next result guarantees the existence of a weak solution to the problem \eqref{Pa}.

\begin{prop} \label{propoprincipal}
Assume the assumptions of Theorem \ref{pricipal-1}. Then $u \in W_0^s L^\Phi (\Omega)$ is a weak solution for \eqref{Pa}, that is,
$$ (1-s)\iint\limits_{\mathbb{R}^N \times \mathbb{R}^N}\phi(|D_s u|)D_s u D_s v\, \mathrm{d}\mu =\int\limits_{\Omega}u^{-\gamma}v\,dx,\quad \forall v \in W^s_0 L^\Phi(\Omega).$$	
\end{prop}
\begin{proof}

For each $t>0$ and $v \in W_0^sL^\Phi(\Omega), v \geq 0$, define the function $\displaystyle \tilde{v}_t=(1-2t)u +t(v-u) \in W_0^s L^\Phi(\Omega)$. In a similar way to those done to obtain \eqref{finite} and \eqref{3.17}, we are able  to show that
\begin{eqnarray}\label{A}
\displaystyle \limsup_{t \rightarrow 0}\frac{1}{t}  \int\limits_\Omega \frac{\tilde{v}_{t}^{1-\gamma}-u^{1-\gamma}}{1-\gamma}\mathrm{d}x  
&\leq & \displaystyle(1-s)\iint\limits_{\mathbb{R}^N \times \mathbb{R}^N}\phi(|D_s u|)D_s uD_s(-3u+v )\mathrm{d}\mu\nonumber \\
&\leq& \displaystyle\lambda\iint\limits_{\mathbb{R}^N \times \mathbb{R}^N}\widetilde{\Phi}(\phi(|D_s u|)|D_s u|)\mathrm{d}\mu+ \iint\limits_{\mathbb{R}^N \times \mathbb{R}^N}\Phi\left(\frac{|D_s(-3u+v )|}{\lambda}\right)\mathrm{d}\mu< +\infty,\nonumber\\
\end{eqnarray}
where the last inequality follows from the use of Young's inequality, $\phi(|D_s u|)|D_s u|\in D_{\tilde{\Phi}}$ and the fact that 
$$\displaystyle \iint\limits_{\mathbb{R}^N \times \mathbb{R}^N}\Phi\left(\frac{|D_s(-3u+v )|}{\lambda}\right)\mathrm{d}\mu<\infty$$
for $\lambda>1$ big enough due to $-3u+v\in W_0^sL^\Phi(\Omega)$.

 Reminding that $0<\tilde{v}_t=(1-2t)u +t(v-u) = u-t(3u-v) \in W_0^sL^\Phi(\Omega)$ for
$ t<{1}/{3}$, $u^{-\gamma}\tilde{v}_t \in L^1(\Omega)$ by Lemma \ref{solpos}, and following similar arguments as done to prove  \eqref{vmaioru1},  we have 
\begin{eqnarray*}\label{B}
\begin{array}{lclll}
    \displaystyle\limsup_{t \rightarrow 0}\frac{1}{1-\gamma}\int\limits_\Omega \frac{\tilde{v}_t^{1-\gamma} -u^{1-\gamma}}{t}\mathrm{d}x 
    & \geq &  \displaystyle  \int\limits_\Omega (-3u^{1-\gamma}+  u^{-\gamma}v) \mathrm{d}x
\end{array}    
\end{eqnarray*}
holds so that 
$$\displaystyle +\infty > (1-s)\iint\limits_{\mathbb{R}^N \times \mathbb{R}^N}\phi(|D_s u|)D_s u D_sv \mathrm{d}\mu - 3(1-s) \iint\limits_{\mathbb{R}^N \times \mathbb{R}^N}\phi(|D_s u|)|D_s u|^2 \mathrm{d}\mu \geq \int\limits_\Omega u^{-\gamma}v \mathrm{d}x - 3\int\limits_\Omega u^{1-\gamma}\mathrm{d}x,$$
after using \eqref{A}. Just using this inequality combined with  Corollary \ref{igual} and Proposition \ref{impontant}, we obtain
\begin{equation}\label{geq}
    (1-s)\iint\limits_{\mathbb{R}^N \times \mathbb{R}^N}\phi(|D_s u|)D_s u D_s v \mathrm{d}\mu \geq \int\limits_{\Omega}u^{-\gamma}v \mathrm{d}x, \ \ \ \forall v \in W_0^sL^\Phi(\Omega), v \geq 0.
\end{equation}

For any $v \in W_0^sL^\Phi(\Omega)$ given, define  $\bar{v}_t= [(1-2t)u +t(v-u)]^+, \  t > 0$. So, it follows from  Lemma \ref{u+}  that $\bar{v}_t \in W_0^s L^\Phi(\Omega)$, which enable us to test \eqref{geq} against $\bar{v}_t$. Before doing this, let us define $\omega_t:=(1-3t)u +tv$, $t>0$, and denote by $\Omega_t^+:=[\omega_t >0]$, and $\Omega_t^-:=[\omega_t \leq 0]$ to obtain
\begin{eqnarray}\label{trambolho1}
\begin{array}{lcl}
    0 & \leq & \displaystyle (1-s)\iint\limits_{\mathbb{R}^N \times \mathbb{R}^N}\phi(|D_s u|)D_s u D_s([(1-3t)u +tv]^+)\mathrm{d}\mu -\displaystyle \int\limits_{\Omega}u^{-\gamma}([(1-3t)u +tv]^+)\mathrm{d}x \\
    
     & = & \displaystyle  (1-s)\iint\limits_{\Omega_t^+ \times \Omega_t^+}\phi(|D_s u|)D_s u D_s \omega_t \mathrm{d}\mu -  \int\limits_{\Omega_t^+}u^{-\gamma} \omega_t \mathrm{d}x \\
     
     & = & \displaystyle (1-s) \left[ \ \displaystyle \iint\limits_{\mathbb{R}^N \times \mathbb{R}^N} - \iint\limits_{\Omega_t^- \times \Omega_t^-} -\iint\limits_{\Omega_t^+ \times \Omega_t^-} - \iint\limits_{\Omega_t^- \times \Omega_t^+}\right]\phi(|D_s u|)D_s u D_s \omega_t \mathrm{d}\mu -  \left[\int\limits_\Omega - \ \int\limits_{\Omega_t^-} \right]u^{-\gamma}\omega_t \mathrm{d}x \\

     & = & \displaystyle  (1-s)\left[ \ \iint\limits_{\mathbb{R}^N \times \mathbb{R}^N}\phi(|D_s u|)|D_s u|^2 \mathrm{d}\mu + t \iint\limits_{\mathbb{R}^N \times \mathbb{R}^N} \phi(|D_s u|)D_s uD_s v \mathrm{d}\mu \right.\\
     
     &-& \displaystyle \left.3t \iint\limits_{\mathbb{R}^N \times \mathbb{R}^N}\phi(|D_s u|)|D_s u|^2 \mathrm{d}\mu \right] -\displaystyle  \int\limits_{\Omega}u^{1-\gamma}\mathrm{d}x - t \int\limits_{\Omega}u^{-\gamma}v\mathrm{d}x + 3t \int\limits_\Omega u^{1-\gamma}\mathrm{d}x \\
     
     & + & \displaystyle  (3t-1)(1-s)\left[ \ \iint\limits_{\Omega_t^- \times \Omega_t^-} +\iint\limits_{\Omega_t^+ \times \Omega_t^-} +\iint\limits_{\Omega_t^- \times \Omega_t^+} \right] \phi(|D_s u|)|D_s u|^2\mathrm{d}\mu\\
     
     &-& \displaystyle  t(1-s) \left[ \ \iint\limits_{\Omega_t^- \times \Omega_t^-} +\iint\limits_{\Omega_t^+ \times \Omega_t^-} +\iint\limits_{\Omega_t^- \times \Omega_t^+}  \right] \phi(|D_s u|)D_s u D_s v\mathrm{d}\mu +   \int\limits_{\Omega_t^-}u^{-\gamma}\omega_t\mathrm{d}x.
\end{array}
\end{eqnarray}

So, by taking $t<{1}/{3}$ and using Corollary \ref{igual}, we have in \eqref{trambolho1}
\begin{eqnarray}\label{trambolhinho}
\begin{array}{lcl}
    0 & \leq &  \displaystyle   t \left[(1-s) \iint\limits_{\mathbb{R}^N \times\mathbb{R}^N}\phi(|D_s u|)D_s u D_s v \mathrm{d}\mu - \int\limits_\Omega u^{-\gamma}\mathrm{d}x  \right.  \\
    
    & - & \displaystyle \left. (1-s) \Big[ \ \iint\limits_{\Omega_t^- \times \Omega_t^-}+ \iint\limits_{\Omega_t^+ \times \Omega_t^-}+ \iint\limits_{\Omega_t^- \times \Omega_t^+} \Big]\phi(|D_s u|)D_s u D_s v \mathrm{d}\mu\right]~\mbox{for all }t>0.
\end{array}
\end{eqnarray}

By noting that 
$$\lim_{t \rightarrow 0}\chi_{\Omega_t^+}(x)=1, \ \text{a.e. in} \ \Omega \ \text{and} \ \lim_{t \rightarrow 0}\chi_{\Omega_t^-}(y)=0, \ \text{a.e. in} \ \Omega,$$ we have from  Dominated Convergence Theorem that 
$$\displaystyle (1-s)\left[ \ \iint\limits_{\Omega_t^- \times \Omega_t^-} +\iint\limits_{\Omega_t^+ \times \Omega_t^-} + \iint\limits_{\Omega_t^- \times \Omega_t^+}\right]\phi(|D_s u|)D_s u D_s v \mathrm{d}\mu \longrightarrow 0, \ \ \text{as} \ \ t \rightarrow 0, \ \forall v \in  W_0^s L^\Phi(\Omega),$$
whence follows, combined with \eqref{trambolhinho}, that
\begin{equation*}\label{3.20}
0 \leq (1-s) \iint\limits_{\mathbb{R}^N \times \mathbb{R}^N}\phi(|D_s u|)D_s u D_s v \mathrm{d}\mu - \int\limits_{\Omega}u^{-\gamma}v\mathrm{d}x, \ \ \forall v \in  W_0^s L^\Phi(\Omega)
\end{equation*}
holds. As $v \in W_0^sL^\Phi(\Omega)$ was taken arbitrary, we have
\begin{equation*}\label{3.14}
\displaystyle  (1-s)\iint\limits_{\mathbb{R}^N \times \mathbb{R}^N} \phi(|D_s u|)D_s u D_s v \mathrm{d}\mu -\int\limits_\Omega u^{-\gamma}v \mathrm{d}x = 0, \ \ \forall v \in  W_0^s L^\Phi(\Omega),  
\end{equation*}
that is, $u \in W_0^s L^\Phi (\Omega)$ is an weak solution for the problem \eqref{Pa}. This ends the proof.
\end{proof}

In fact, we have uniqueness of the weak solutions to the problem \eqref{Pa}.
\begin{prop}\label{unicidadesol} Assume the assumptions of Theorem \ref{pricipal-1}. Then
$u \in W_0^s L^\Phi(\Omega)$ is the
unique weak solution to problem \eqref{Pa}.
\end{prop}
\begin{proof}
Suppose that $v \in W^s_0 L^\Phi(\Omega)$ is a weak solution for the problem \eqref{Pa} as well. So,
\begin{equation*}\label{u1-u2}
\displaystyle (1-s)\iint\limits_{\mathbb{R}^N \times \mathbb{R}^N}\left[\phi(|D_s u|)D_s u - \phi(|D_s v|)D_s v \right]D_s(u - v) \mathrm{d}\mu = \int\limits_{\Omega} (u^{-\gamma}-v^{-\gamma})(u-v)\mathrm{d}x \leq 0, 
\end{equation*}
where the inequality follows from $\gamma>0$.

On the other hand, as  $\Phi$ is  convex, due to $\phi(|t|)t$ be increasing, we have 
$$\displaystyle 0 \leq (1-s)\iint\limits_{\mathbb{R}^N \times \mathbb{R}^N}\left[\phi(|D_s u|)D_s u - \phi(|D_s v|)D_s v \right]D_s(u - v) \mathrm{d}\mu$$
so that
\begin{equation*}\label{contradic}
\displaystyle \int\limits_{\Omega}(u^{-\gamma}-u_2^{-\gamma})(u-u_2)\mathrm{d}x  = 0 ,  
\end{equation*}
which implies that $u=v$. this finishes the proof.\
\end{proof}

\textbf{Proof of Theorem \ref{pricipal-1}-Conclusion:}
The proof of the Theorem \ref{pricipal-1} follows directly from the Propositions \ref{propoprincipal} and \ref{unicidadesol}.

\textbf{Proof of Theorem \ref{pricipal-2}-Conclusion:} The proof follow the same strategy and arguments as used to prove the Theorem \ref{pricipal-1}.

\section{Proof of Theorem \ref{gammaconvergence}}

The arguments of this section are inspired on those ones found in \cite{FBS}. First, we note that the $N$-function 
\begin{equation}\label{Psi2}
\Psi(t)=\lim_{s \uparrow 1}\int_0^t \int_{\mathbb{S}^{N-1}} \Phi(\rho|z_N|)\mathrm{d}S_z \frac{\mathrm{d}\rho}{\rho}=\int_0^t \int_{\mathbb{S}^{N-1}} \Phi(\rho|z_N|)\mathrm{d}S_z \frac{\mathrm{d}\rho}{\rho},~t\in \mathbb{R},
\end{equation}
defined in \eqref{Psi1}, also satisfies
$$\begin{array}{c}\displaystyle 0\;<\;\Psi^{\prime}(t)\;\leq\;\frac{\Phi(t)}{t}\;|\mathcal{S}^{N-1}|,~t>0
\end{array} 
$$
and
$$\begin{array}{c}\displaystyle\;\;\frac{(\ell-1)}{t^2}\int_{\mathcal{S}^{N-1}}\Phi(t|z_{N}|)d\mathcal{S}_{N}\;\;\leq\;\Psi^{\prime\prime}(t)\;\leq\;(m-1)\left|\mathcal{S}^{N-1}\right|\;\frac{\Phi(t)}{t^2}\;
\end{array}~t>0, $$ 
whence follows, in particular, that $\Psi$ is strictly convex.

We already know from \cite[Proposition 2.16]{FBS} the below Proposition.
\begin{prop}
 $\Psi$ is an $N$-function. 
Furthermore, there exist $k_1, k_2 >0$ such that $$k_1\Phi(t) \leq \Psi(t) \leq k_2\Phi(t), \ \forall t>0,$$
that is, the N-function $\Psi$ is equivalent to $N$-function $\Phi$.
\end{prop}

It is worth mentioning that as in \cite[Theorem 4.1]{FBS}, we have the following result that explicitly shows the limit of the operator of $(-\Delta_\Phi)^s$ as $s \uparrow 1$. The validity of this theorem  for the case $\ell = 1$ is also proved in  \cite[Theorem 1.1 and Remark 1.2]{Cianchi}.

\begin{thm}\label{limitoperator}
Let $\Phi$ be an $N$-function such that $\Phi\in \Delta_2$. Then, given $u \in L^\Phi(\Omega)$ and $0<s<1$ it holds that 
$$\lim_{s \uparrow 1}I_1(u)=\int\limits_{\Omega}\Psi(|\nabla u|)\mathrm{d}x,$$
where $\Psi$ is defined in \eqref{Psi1}.
\end{thm}

Now, note that as a consequence of  $\displaystyle \lim_{t \rightarrow +\infty}{\Phi(t)}/{t}=+\infty$, there  exists a $T>1$ big enough such that $$\displaystyle \frac{1}{2C}\Phi(t) > t, \ \forall ~t \geq T,$$ where $C>0$ is the Poincaré's  constant. So, by using  $I(u_s)<0$ (see Lemma \ref{negative-sing}), we have 
\begin{eqnarray}\label{desig}
    \begin{array}{lclll}
    \displaystyle (1-s) \iint\limits_{\mathbb{R}^N \times \mathbb{R}^N} \Phi(|D_s u_s|) \mathrm{d}\mu  & \leq & \displaystyle \frac{1}{1-\gamma}\int\limits_\Omega |u_s|^{1-\gamma} \mathrm{d}x   & \leq & \displaystyle T|\Omega| + \int\limits_{[u_s > T]}|u_s| \mathrm{d}x \\
    
    & \leq & \displaystyle T|\Omega| + \frac{1}{2C}\int\limits_\Omega \Phi(|u_s|)\mathrm{d}x & \leq & \displaystyle T|\Omega| + \frac{(1-s)}{2} \iint\limits_{\mathbb{R}^N \times \mathbb{R}^N} \Phi(|D_s u_s|)\mathrm{d}\mu,
    \end{array}
\end{eqnarray}
where, the second inequality in the first line above is a consequence of Young's inequality. After  \eqref{desig}, we are able to prove the below result whose proof follows similar arguments as those done in \cite[Theorem 5.1]{FBS}.

\begin{thm}\label{caseofasequence}
Assume the hypotheses of Theorem \ref{gammaconvergence}. Let $0 < s_n \uparrow 1$ and $\{ u_n\}_{n \in \mathbb{N}} \subset L^\Phi(\mathbb{R}^N)$ such that 
$$\displaystyle \sup_{n \in \mathbb{N}}(1-s_n)\iint\limits_{\mathbb{R}^N \times \mathbb{R}^N}\Phi(|D_s u_n|)\mathrm{d}\mu < +\infty \ \ \ \text{and} \ \ \ \sup_{n \in \mathbb{N}}\int\limits_\Omega \Phi(u_n)\mathrm{d}x < +\infty.$$
Then there exists $u \in L^\Phi(\Omega)$ and a subsequence $\{ u_{n_j} \}_{j \in \mathbb{N}} \subset \{ u_n \}_{n \in \mathbb{N}}$ such that $u_{n_j} \rightarrow u$ in $L^\Phi_{loc}(\mathbb{R}^N)$. Moreover, $u \in W^1 L^\Phi(\mathbb{R}^N)$ and holds $$\displaystyle \rho_\Psi (\nabla u) \leq \liminf_{n \rightarrow +\infty} (1-s_n)\iint\limits_{\mathbb{R}^N \times \mathbb{R}^N}\Phi(|D_s u_n|)\mathrm{d} \mu.$$
\end{thm}

To complete our arguments, let us remember the following concepts and results.
\begin{defi} Let $X$ be a metric space and $F, F_{j}: X \longrightarrow \overline{\mathbb{R}}$. We say that \textbf{\textit{$F_{j}$ $\Gamma$-converges to $F$}} if for every $u \in X$ we have

\begin{enumerate}
    \item For every sequence $\left\{u_j \right\}_{j \in \mathbb{N}} \subset X$ such as $u_j\rightarrow u$ in $X$ $$F(u) \leq \liminf_{j \rightarrow \infty}F_{j}(u_j);$$
    \item For every $u \in X$, exists $\left\{u_j \right\}_{j \in \mathbb{N}} \subset X$ converging to $u$ such that
    $$F(u) \geq \limsup_{j \rightarrow \infty}F_{j}(u_j).$$
\end{enumerate}

The functional \textbf{\textit{$F$ is the $\Gamma$-limit of the sequence $\left\{F_j \right\}_{j \in \mathbb{N}}$}} and it is denoted by $F_j \stackrel{\Gamma}{\rightarrow} F$ and $$F = \Gamma- \lim_{j \rightarrow \infty} F_j.$$
\end{defi}

Is important to mention that the main feature of the $\Gamma$-convergence is the convergence of minima of suitable functionals, as  proved in \cite[Theorem 6.6]{FBS} and stated below.
\begin{thm}
Let $(X, d)$ be a metric space and let $F,Fj:X\longrightarrow \overline{\mathbb{R}}, \ j \in \mathbb{N}$, be such that $F_j\stackrel{\Gamma}{\rightarrow} F$. Assume that for each $j\in \mathbb{N}$ there exist $u_j \in X$ such that $F_j(u_j) =\inf_{X} F_j$ and suppose that the sequence $\left\{u_j\right\}_{j\in \mathbb{N}} \subset X$ is precompact.
Then every accumulation point of $\left\{u_j\right\}_{j\in \mathbb{N}}$ is a minimum of F and $$\inf_{X} F =  \lim_{j \rightarrow \infty} \inf_{X} F_j.$$
\end{thm}

After these, let us define  
$$
\displaystyle \mathcal{J}_s(u) = \left\{\begin{array}{lll}I_1(u)  & u \in W_0^sL^\Phi(\Omega);&  \\
\displaystyle +\infty, &\text{otherwise}&
\end{array}
\right.
\qquad\mbox{and}\qquad \ \ \ 
\displaystyle \mathcal{J}(u) = \left\{\begin{array}{lll}\rho_\Psi(|\nabla u|)  & u \in W_0^1L^{\tilde{\Phi}}(\Omega);&  \\
\displaystyle +\infty, &\text{otherwise}&
\end{array}
\right.
$$
and point out that $\mathcal{J}_s\stackrel{\Gamma}{\to} \mathcal{J}$, as done in the \cite[Theorem 6.5]{FBS}. So, by using this convergence together with the fact that  $I_2$ is continuous in $L^\Phi(\Omega)$ (see Lemma \ref{I2continuous}), we have that
$\mathcal{F}_s\stackrel{\Gamma}{\to} \mathcal{F}$ as well, as done in \cite{FBS}, where 
 $$\displaystyle \mathcal{F}_s (u) = \mathcal{J}_s
(u) - I_2(u) \ \ \ \text{and} \ \ \   \mathcal{F} (u) = \mathcal{J}
(u) -I_2(u).$$

As proved in  Chapter \ref{cap-existence}, the weak solution of Problem \eqref{Pa} $u=u_s \in W_0^{1,\Phi}(\Omega)$ is the unique  global minimizer of $\mathcal{F}_s$, that is, $$\displaystyle \mathcal{F}_s (u_s) = \inf_{v \in L^\Phi(\Omega)} \mathcal{F}_s (v),$$
so that we are in the same context of \cite[Lemma 6.7]{FBS}.

As a consequence of the above structure and Theorem \ref{caseofasequence}, we have the following result.
\begin{thm}
Let $0 < s_j \uparrow 1$ and $\Omega \subset \mathbb{R}^N$ be a open bounded subset. Given $j \in \mathbb{N}$, let $u_j \in L^\Phi(\Omega)$ be the minimum of $\mathcal{F}_{s_j}$. Then $\{u_j \}_{j \in \mathbb{N}} \subset L^\Phi(\Omega)$ is precompact.
\end{thm}

As a corollary of the results above, we obtain our statement.
\begin{thm}\label{limit}
Let $\Phi$ in the hypotheses of Theorem \ref{gammaconvergence} and $u_s \in L^\Phi(\Omega)$ the minimun of $\mathcal{F}_{s_j}$. Then there exists $u \in L^\Phi(\Omega)$ such that $$u=\lim_{s \uparrow 1}u_s \ \text{in} \ L^\Phi(\Omega) \ \ \ \text{and} \ \ \ \mathcal{F}(u) = \min_{v \in L^\Phi(\Omega)} \mathcal{F}(v).$$
\end{thm}
\textbf{Proof of Theorem \ref{gammaconvergence}-Conclusion:} By Theorem \ref{pricipal-2}, Problem \eqref{Psilaplacian} has a unique positive weak solution $u \in W_0^1  L^\Phi(\Omega)$, which is the critical point of the functional $\mathcal{F}$. Then, the proof of the Theorem \ref{gammaconvergence} follows by applying  Theorem \ref{limit}.

\section{Conclusion}
In this paper, we need to consider the hypotheses  $(\phi_1)$, $(\phi_2)$ and $\Phi\in \Delta_2$ to prove Theorem \ref{gammaconvergence}. We believe that Theorem \ref{gammaconvergence} will continue valid without assuming the hypothesis $\Phi\in \Delta_2$ by changing the strong convergence to the modular convergence. Spite of classical arguments do not work in our case anymore, we also believe that the solutions obtained by Theorems \ref{pricipal-1} and \ref{pricipal-2} are in  $L^\infty(\Omega)$ just assuming $(\phi_1)$, $(\phi_2)$ and $\Phi\in \Delta_2$.

\section{Declarations}
{\bf Author Contributions:} Marcos L. M. Carvalho, Luana C. M. Lima, Carlos A. P. Santos, and Maxwell L. Silva contributed equally to this work and should be considered co-first authors.
\\

{\bf Ethical Approval:} Not applicable.
\\

{\bf Competing interests:} The authors declare no competing interests.
\\

{\bf Funding:} The first author was also partially supported by CNPq with the grant 300411/2025-1.\\
The second author was also partially supported by CAPES with the grant 88882.386249/2019-01.\\
The third author was also partially supported by CNPq with the grant 311562/2020-5, and FAPDF under the grant 00193.00001133/2021-80.
\\

{\bf Data Availability:} No datasets were generated or analyzed during the current study

\end{document}